\documentclass[12pt]{amsart} 
\usepackage[mathscr]{eucal} 
\usepackage{amsmath,amsfonts} 
 \usepackage{hyperref}
 
\parskip=\smallskipamount 
 
\newtheorem{theorem}{Theorem}[section] 
\newtheorem{proposition}[theorem]{Proposition} 
\newtheorem{corollary}[theorem]{Corollary} 
\newtheorem{lemma}[theorem]{Lemma} 
 
\theoremstyle{definition} 
\newtheorem{definition}[theorem]{Definition} 
\newtheorem{example}[theorem]{Example}

\newtheorem{remark}[theorem]{Remark}

\makeatletter 
\@addtoreset{equation}{section} 
\makeatother

\newcommand{\CC}{{\mathbb C}} 
\newcommand{\FF}{{\mathbb F}}

\newcommand{\NN}{{\mathbb N}}

\newcommand{\RR}{{\mathbb R}} 
 
\newcommand{\cB}{{\mathcal B}} 
 
\newcommand{\cD}{{\mathcal D}} 
 
\newcommand{\cF}{{\mathcal F}} 
\newcommand{\cG}{{\mathcal G}} 
\newcommand{\cH}{{\mathcal H}} 
 
\newcommand{\cK}{{\mathcal K}} 
\newcommand{\cL}{{\mathcal L}} 
 
\newcommand{\cN}{{\mathcal N}}

\newcommand{\cR}{{\mathcal R}}

\newcommand{\cV}{{\mathcal V}} 

\newcommand{\cX}{{\mathcal X}} 
 
\newcommand{\bD}{\boldsymbol{D}}
\newcommand{\bG}{\boldsymbol{G}}
\newcommand{\bH}{\boldsymbol{H}}
\newcommand{\bK}{\boldsymbol{K}}
\newcommand{\bR}{\boldsymbol{R}}

\newcommand{\bU}{\boldsymbol{U}}
\newcommand{\bV}{\boldsymbol{V}}

\newcommand{\fK}{\mathfrak{K}}

\newcommand{\dom}{\operatorname{Dom}} 
 
\newcommand{\Ra}{\Rightarrow} 
\newcommand{\ran}{\operatorname{Ran}} 
 
\newcommand{\ra}{\rightarrow} 
\newcommand{\ol}{\overline}

\let\phi=\varphi 
\newcommand{\iac}{\mathrm{i}} 
\renewcommand{\ker}{\operatorname{Ker}}

\newcommand{\lin}{\operatorname{Lin}}

\newcommand{\Clos}{\operatorname{Clos}}
\newcommand{\Span}{\operatorname{Span}}
\renewcommand{\Re}{\operatorname{Re}}
\renewcommand{\Im}{\operatorname{Im}}

\newcommand{\nr}[1]{\vspace{0.1ex}\noindent\hspace*{12mm}\llap{\textup{(#1)}}} 
 
\title[Operator Valued Kernels]{Localised Operator Valued Kernels Invariant under Actions of $*$-Semigroupoids}
  
\dedicatory{To the memory of Heinz Langer, coauthor, mentor, and friend, who had one of the strongest influences on the 
development of operator theory in indefinite inner product spaces}
  
\author[A. Gheondea]{Aurelian Gheondea} 

\address{Institute of Mathematics of the Romanian Academy,
  Calea Grivi\c tei 21, 010702 Bucure\c sti, Rom\^ania, \emph{and}
  Department of Mathematics, Bilkent University, 06800 Bilkent, Ankara, 
      {\protect T\"urk\.iye}}
\email{A.Gheondea@imar.ro \textrm{and} aurelian@fen.bilkent.edu.tr} 

\begin{document}

\begin{abstract} We consider positive semidefinite kernels which have values given by bounded linear operators on certain 
bundles of Hilbert spaces and which are invariant under actions of $*$-semigroupoids. For these kernels, we prove that 
there exist generalised $*$-representations of the $*$-semigroupoids 
on the underlying reproducing kernel Hilbert spaces or, equivalently, on the 
underlying minimal linearisations, we characterise when the $*$-representations are performed 
by means of bounded operators and show that this always happens for inverse semigroupoids. Then, we consider 
Hermitian kernels which have values given by bounded linear operators on certain 
bundles of Hilbert spaces and which are invariant under actions of $*$-semigroupoids. Only those Hermitian kernels having 
certain boundedness properties can produce reproducing kernel Krein spaces but uniqueness is more complicated. However, for 
these kernels, generalised $*$-representations can be obtained. If $*$-representations with bounded linear operators 
are requested, then stronger boundedness conditions on the kernels are needed.
\end{abstract}


\maketitle

\section{Introduction}

Positive semidefinite scalar kernels have been considered since the beginning of the twentieth century, in connection with partial 
differential equations by S.~Zaremba \cite{Zaremba}, integral equations by J.~Mercer \cite{Mercer}, spaces of analytic 
functions by S.~Szeg\"o \cite{Szego}, S.~Bergman \cite{Bergman}, S.~Bochner \cite{Bochner1}, general analysis questions by
E.H.~Moore \cite{Moore}, stationary sequences by N.~Kolmogorov \cite{Kolmogorov}, until two general formalisations were 
performed by N.~Aronszajn \cite{Aronszajn} and L.~Schwartz \cite{Schwartz}. 
Then, this theory was extended to operator valued kernels, firstly by G.~Pedrick \cite{Pedrick} to locally convex spaces, then
by R.M.~Loynes \cite{Loynes1}, \cite{Loynes2}, \cite{Loynes3}, to vector valued Hilbert spaces, and so on. In the meantime,
this theory was used in many other domains of mathematics, 
such as: the theory of group representations by M.A.~Naimark \cite{Naimark1}, \cite{Naimark2},
 R.~Godement \cite{Godement}, S.~Bochner \cite{Bochner}, M.G.~Krein 
\cite{Krein1}, \cite{Krein2}, M.B.~Bekka and P.~de la Harpe \cite{Bekka}, function spaces by L.~de~Branges \cite{deBranges1},
canonical models in quantum mechanics by L.~de~Branges and J.~Rovnyak 
\cite{deBrangesRovnyak}, dilation theory by W.F.~Stinespring \cite{Stinespring}, B.~Sz.-Nagy \cite{BSzNagy}, 
D.E.~Evans and J.T.~Lewis \cite{EvansLewis}, T.~Constantinescu 
\cite{Constantinescu}, probability theory by K.R.~Parthasarthy \cite{Parthasarathy1}, \cite{Parthasarathy}, 
K.R.~Parthasarathy and K.~Schmidt \cite{ParthasarathySchmidt}, Hilbert $C^*$-modules by G.J.~Murphy \cite{Murphy}, and
Hilbert modules over locally $C^*$-algebras by the author \cite{Gheondea5}. 
A partial image of the big picture on the theory and the many
applications of the theory of positive semidefinite kernels can be seen in the monographs of 
A.~Berlinet and C.~Thomas-Agnan \cite{Berlinet}, V.R.~Paulsen and M.~Raghupathi \cite{PaulsenRaghupathi}, and
S.~Saitoh and Y.~Sawano \cite{SaitohSawano}. 

Reproducing kernel Pontryagin spaces first appeared at Sorjonen \cite{Sorjonen}.
A very fruitful research direction in the theory of kernels was opened by the seminal series of five articles 
of M.G.~Krein and H.~Langer \cite{KreinLanger1}, \cite{KreinLanger2},   \cite{KreinLanger3}, \cite{KreinLanger4},
 \cite{KreinLanger5}, in which Hermitian scalar kernels which are not semidefinite have been considered in connection with 
generalisations of the classical kernels of C.~Caratheodory, G.~Szeg\"o, and R.~Nevanlinna. These investigations triggered further research for the special case of Pontryagin spaces, where the existence 
and uniqueness are very similar to the case of Hilbert spaces, as can be seen in 
the monograph  of D.~Alpay, A.~Dijksma, J.~Rovnyak, H.~de Snoo \cite{AlpayDijksmaRovnyakDeSnoo}. 
For the more general case of Krein spaces, the situation is technically rather different. 
Here, it should be pointed out that, in the seminal paper of L.~Schwartz \cite{Schwartz}, reproducing kernel spaces, in his special 
interpretation of Hilbert spaces continuously contained in a given quasi-complete locally convex space, have been generalised to 
Krein spaces, which he called Hermitian spaces.
A related fact 
concerning realisations of Hermitian analytic kernels as operator colligations, also called transfer functions, in Krein spaces was 
studied by A.~Dijksma, H.~Langer, and H.S.V. de Snoo in \cite{DLS1}. D.~Alpay \cite{Alpay} 
showed that analyticity of the kernel is a sufficient condition for existence of the reproducing kernel Krein space, while, in general, 
as can be seen at L.~Schwartz \cite{Schwartz} and in the series of articles of 
T.~Constantinescu with the author \cite{ConstantinescuGheondea}, \cite{ConstantinescuGheondea0}, 
\cite{ConstantinescuGheondea1}, \cite{ConstantinescuGheondea2}, additional boundedness conditions are necessary. The 
result of Alpay was generalised to several complex variables in \cite{ConstantinescuGheondea3}.

In the last twenty years, positive semidefinite kernels and their reproducing kernel Hilbert spaces were 
found very useful in machine 
learning, as can be seen, for example, in the monographs of I.~Steinwart and C.~Christman \cite{SteinwartChristman} and that of  
B.~Sch\"olkopf and A.J. Smola \cite{ScholkopfSmola}. Then, 
operator valued kernels turned out to be extremely useful in learning theory \cite{Caponnetto}, \cite{Carmelli}, 
\cite{GheondeaTilki}, \cite{Hashimoto21},  \cite{Hashimoto22}, \cite{Hashimoto23a}, \cite{Hashimoto23b}, \cite{Hashimoto24},
\cite{Kadri}, \cite{Lambert}, \cite{Micchelli}, \cite{MBM},  among many other works. Moreover, recently, Hermitian kernels and their 
reproducing kernel Krein spaces have been found useful in machine learning \cite{OglicGartner}, \cite{Ong}.
It has been understood for some good time that kernels that are invariant under actions of certain 
semigroups or groups turned out to cover many other mathematical problems and were found to 
be useful in many applications, as pointed out in \cite{Bekka}, \cite{EvansLewis},
\cite{ParthasarathySchmidt}, \cite{Constantinescu}, \cite{ConstantinescuGheondea2}, \cite{Gheondea}, 
\cite{AyGh2015}, \cite{AyGh2017}, \cite{AyGh2019}. 

In view of the recent interest in the theory of groupoids in functional analysis, which was triggered by the investigations of
J.~Cuntz and W.~Krieger \cite{CuntzKrieger}, operator valued positive semidefinite maps on $*$-semigroupoids have been 
considered by B.~Udrea and the author in \cite{GheondeaUdrea}. It is known that operator valued kernels that are invariant 
under left actions of $*$-semigroups are models of a system that shows certain symmetry properties and, because of that, it 
contains many other problems: operator valued completely positive maps on $C^*$-algebras,
operator valued maps on groups which are positive semidefinite, semispectral measure, and many others.  It is expected that this 
feature will be preserved when $*$-semigroupoids are used instead of $*$-semigroups. From this point of view, 
the concept of partial dynamical system of  R.~Exel \cite{Exel} might be related with actions of semigroupoids. Another motivation 
for this research comes from the recent article \cite{GheondeaTilki} of the author and 
C.~Tilki in which it is pointed out that certain applications of operator valued kernels to machine learning require that the 
operator valued kernels be localised, that is, given a base set $X$ of inputs, to each $x\in X$ one should consider a Hilbert space 
$\cH_x$ that depends on $x$. This setting makes it natural to consider Hilbert space bundles over the input set $X$ which, in 
general, are not trivial in the sense of vector bundles and then, symmetry properties of the system require
partial actions, hence 
semigroupoids become the necessary tool.

In this article, we consider positive semidefinite kernels which have values given by bounded linear operators on certain 
bundles of Hilbert spaces and which are invariant under actions of $*$-semigroupoids. Many of the results that we obtain here 
are generalisations of some of the results obtained by T.~Constantinescu and the author in \cite{ConstantinescuGheondea1} 
and \cite{ConstantinescuGheondea2}, but there are significant differences. Because a left action of a 
$*$-semigroupoid $\Gamma$ on the set $X$ splits the set into a partition, kernels on the set $X$ are also split into a bundle of 
kernels and, consequently, positive semidefiniteness becomes partial.
For these kernels, we prove that there exist generalised, in the sense that the operators on the representation spaces 
are unbounded, $*$-representations of the $*$-semigroupoids 
on the underlying bundles of reproducing kernel Hilbert spaces, cf.\ Theorem~\ref{t:invariantu}, or, equivalently, on the 
underlying bundles of minimal linearisations, cf.\ Theorem~\ref{t:invariantul}. 
Also, we characterise when the $*$-representations are performed 
by means of bounded operators, cf.\ Theorem~\ref{t:invariantb}, where the class of operator valued kernels of bounded shift type 
shows its importance. Then, we show that in the case of inverse semigroupoids, hence in the particular case of groupoids, the 
boundedness condition is automatic, cf.\ Corollary~\ref{c:invsem}. 

In Section~\ref{s:hk}, we consider partial Hermitian kernels which have values given by bounded linear operators on certain 
bundles of Hilbert spaces and investigate the existence of Krein space linearisations, cf.\ Theorem~\ref{t:ksl}, 
and reproducing kernel Krein spaces, cf.\ Theorem~\ref{t:rkks}. Then we consider operator valued Hermitian kernels
which are invariant under actions of $*$-semigroupoids. Compared to  the 
case of positive semidefinite kernels, the difficulties that we encounter now become bigger: 
only those Hermitian kernels having 
certain boundedness properties can produce reproducing kernel Krein spaces and uniqueness is much 
more complicated, cf.\ Theorem~\ref{t:kslp}. However, as 
in the case of positive semidefinite kernels, for these Hermitian kernels, generalised $*$-representations can be obtained,
cf.\ Theorem~\ref{t:invariantuk}. 
If $*$-representations with bounded linear operators are requested, then stronger boundedness conditions, in terms of 
majorisations of the Hermitian kernels with certain partially positive semidefinite kernels of bounded shift type, 
are needed, cf.\ Theorem~\ref{t:invariantukbp}.
In the special case when the majorisation of the Hermitian kernel is performed with  partially positive semidefinite kernels of 
bounded shift type which are also invariant, the corresponding Krein space linearisations and 
reproducing kernel Krein spaces have special joint fundamental reducibility properties, cf.\
Theorem~\ref{t:invariantukb}.

Although the investigations in \cite{GheondeaUdrea} triggered our investigations from this article, the generalisations that we 
obtain to localised operator valued partially positive semidefinite kernels, and which are partially invariant under left actions of 
$*$-semigroupoids, do not fully cover the case of localised partially positive semidefinite maps on $*$-semigroupoids from
\cite{GheondeaUdrea}, because of many differences. Each of the two approaches covers some applications, with advantages 
and disadvantages, when compared with each other. For example, the results proven in this article have the 
advantage that they do not require that the $*$-semigroupoids have units, but the aggregation feature is missing. 
From this point of view, it remains to study in another 
paper the localised  partially Hermitian maps on $*$-semigroupoids and investigate when they produce $*$-representations of 
the underlying semigroupoids on bundles of Krein spaces and of which kind.

In order to make this article selfcontained and because here we use many concepts and results 
on localised operator valued kernels, $*$-semigroupoids and 
their representations, as well as Krein spaces and their linear operators, we included subsections that recall the necessary 
results and make clear the concepts and notation.

\section{Positive Semidefinite Operator Valued Kernels}\label{s:psovk}

\subsection{Notation and Basic Definitions.}\label{ss:nbd}
Let $X$ be a nonempty set and $\{\cH_x\}_{x\in X}$ be a bundle of
Hilbert spaces over $\FF$, either $\RR$ or $\CC$,
with base $X$, that is, $\cH_x$ is a Hilbert space over 
$\FF$, for any $x\in X$. Equivalently, we can use the notation
\begin{equation}\label{e:xast}
X\ast \bH :=\{(x,h)\mid x\in X,\ h\in \cH_x\}
\end{equation}
and let $\pi\colon X\ast \bH\ra X$ denote the canonical projection $\pi(x,h)=x$ for all $(x,h)\in X\ast \bH$. The 
pair $(X\ast \bH,\pi)$, or simply $X\ast \bH$, is called a \emph{Hilbert space bundle over $X$}. Technically, a
vector bundle of Hilbert spaces $X\ast \bH$ is same with the \emph{disjoint  union} of the family of Hilbert 
spaces $\bH$, denoted by $\bigsqcup_{x\in X}\cH_x$. This means that we identify the component $\cH_x$ of the 
disjoint union with the fibre $\{x\}\times \cH_x$ from $X\ast \bH$.
If $\cH_x=\cH$ for all $x\in X$ then $X\ast \bH=X\times \cH$ and, in this case, the bundle $X\times \cH$ is called \emph{trivial}.

A \emph{section}, or a \emph{vector field}, of the bundle, denoted by $f\colon X\ra X\ast \bH$, is a right inverse 
of $\pi$, that is, $\pi\circ f=\mathrm{Id}_X$. This means that $f(x)\in \cH_x$ for all $x\in X$. We will use alternatively 
$f_x=f(x)$, depending on the context and in order to simplify and/or clarify different aspects.
Technically, the 
collection of all sections of the 
bundle $X\ast \bH$, denoted by $\cF_\FF(X;\bH)$, coincides with the product of the family of Hilbert spaces $\bH$, 
that is, $\prod_{x\in X}\cH_x$, which is the collection of all maps $f\colon X\ra  \bigsqcup_{x\in X}\cH_x$ such 
that $f(x)\in \cH_x$ for all $x\in X$. Here we again identify $\cH_x$ with the fibre $\{x\}\times \cH_x$. This is important 
since $\cF_\FF(X;\bH)$ is naturally organised as a vector space, with respect to pointwise addition and multiplication 
with scalars. With respect to this, we identify $f\in \cF_\FF(X;\bH)$ with $(f_x)_{x\in X}$, where $f_x\in \cH_x$ for all 
$x\in X$. With this identification, $\cF_\FF(X;\bH)$ has the product topology, more precisely, for each $x\in X$ 
consider the seminorm
\begin{equation}\label{e:pexef}
p_x(f)=\|f(x)\|_{\cH_x},\quad f\in \cF_\FF(X;\bH).
\end{equation}
With respect to the family of seminorms $\{p_x\mid x\in X\}$, $\cF_\FF(X;\bH)$ becomes a
complete locally convex space. This topology coincides with the topology of pointwise convergence of nets in $\cF_\FF(X;\bH)$.

An \emph{$\bH$-operator valued 
kernel} $K$ is a mapping defined on $X\times X$ such that 
$K(y,x)\in\cB(\cH_x,\cH_y)$ for all $x,y\in X$. We denote by $\fK(X;\bH)$ the collection of all $\bH$-operator
valued kernels on $X$ and it is clear that $\fK(X;\bH)$ is a  vector space over $\FF$.
Given $K\in\fK(X;\bH)$, the adjoint kernel $K^*$ is defined 
by $K^*(x,y)=K(y,x)^*$, for all $x,y\in X$. Clearly, $K^*\in\fK(X;\bH)$. The kernel $K$ is called \emph{Hermitian} if 
$K=K^*$. Any kernel $K$ is a linear combination of two Hermitian kernels, more precisely, letting
\begin{equation}\label{e:rem}\Re(K):=(K+K^*)/2,\quad \Im(K):=(K-K^*)/2\iac,\end{equation}
we have
\begin{equation}\label{e:remek}
K=\Re(K)+\iac \Im(K).
\end{equation}
It is easy to see that $K\mapsto K^*$ is an involution, that is, it is conjugate linear and involutive.
In this way, $\fK(X;\bH)$ is a $*$-vector space. We denote by $\fK^h(X;\bH)$ the real vector space of all 
Hermitian $\bH$-operator valued kernels.

For each $x\in X$ and $h\in\cH_x$, we consider the section 
$\widehat h\in\cF_\FF(X;\bH)$, defined by
\begin{equation}\label{e:deltax}
(\widehat h)(y)=\begin{cases}  h,& \mbox{ if }y=x,\\ 0_{\cH_y},& \mbox{ otherwise.}
\end{cases}
\end{equation}
With a slight abuse of notation, for any $x\in X$, letting $\delta_x\colon X\ra \FF$ be the function $\delta_x(x)=1$ and 
$\delta_x(y)=0$ for $y\neq x$, for any $h\in \cH_x$, we can write
\begin{equation}\label{e:delta}
\widehat h=\delta_x h.
\end{equation}
This notation is useful under certain circumstances since it emphasises the role of $x$.
Then, for any $f\in\cF_\FF(X;\bH)$, we have
\begin{equation}\label{e:fesum}
f=\sum_{x\in X}  \widehat f_x=\sum_{x\in X} \delta_x f_x,
\end{equation}
where $f_x:=f(x)$ for all $x\in X$. Clearly, the convergence in \eqref{e:fesum} is pointwise.

Let $\cF_\FF^0(X;\bH)$ be the vector subspace consisting of all $f\in\cF_\FF(X;\bH)$ with 
finite support. Clearly, any section of type $\widehat h$ belongs to $\cF_\FF^0(X;\bH)$ and, for 
any $f\in\cF_\FF^0(X;\bH)$ there exists unique distinct elements $x_1,\ldots,x_n\in X$ and 
$h_i=f(x_i)\in\cH_{x_i}$, $i=1,\ldots,n$, such that 
\begin{equation*} f=\sum_{i=1}^n \widehat h_i=\sum_{i=1}^n \delta_{x_i}h_i.
\end{equation*}
 An inner product $\langle\cdot,\cdot\rangle_0\colon \cF_\FF^0(X;\bH)\times 
\cF_\FF^0(X;\bH)\ra \FF$ can be defined by
\begin{equation}\label{e:ipzero} \langle f,g\rangle_0=\sum_{x\in X}\langle f(x),g(x)
\rangle_{\cH_x},
\quad f,g\in \cF_\FF^0(X;\bH).
\end{equation}
In addition, 
let us observe that the sum in \eqref{e:ipzero} makes sense in the more general case when $f,g\in\cF_\FF(X;\bH)$ 
and at least one of $f$ or $g$ has finite support, the other can be arbitrary.  

Associated to the kernel $K\in\fK(X;\bH)$ there is a 
sesquilinear / bilinear form $\langle \cdot,\cdot\rangle_K\colon \cF_\FF^0(X;\bH)\times 
\cF_\FF^0(X;\bH)\rightarrow \FF$ defined by
\begin{equation}\label{e:ipk} \langle f,g\rangle_K=\sum_{x,y\in X}\langle K(y,x) 
f(x),g(y)\rangle_{\cH_y},\quad f,g\in \cF_\FF^0(X;\bH),
\end{equation}
that is, $\langle\cdot,\cdot\rangle_K$ is linear in the first variable 
and conjugate linear / linear in the second variable. Also, the form $\langle\cdot,\cdot
\rangle_K$ is Hermitian, that is, $\langle f,g\rangle_K=\overline{\langle g,f\rangle_K}$ for all 
$f,g\in \cF_\FF^0(X;\bH)$, if and only if the kernel $K$ is Hermitian. 

A \emph{convolution operator} $C_K\colon \cF_\FF^0(X;\bH)\ra \cF_\FF(X;\bH)$ can be defined by
\begin{equation}\label{e:conv} (C_Kf)(y)=\sum_{x\in X} K(y,x)f(x),\quad f\in
\cF_\FF^0(X;\bH),\ y\in X.
\end{equation}
Clearly, $C_K$ is a linear operator and, for all $x,y\in X$ and all $h\in\cH_x$, we have
\begin{equation}\label{e:cekaw}
(C_K \widehat h)(y)=(C_K \delta_x h)(y)= \sum_{z\in X} K(y,z)\delta_x(z) h=K(y,x)h.
\end{equation}
Introducing the notation, for $x\in X$, $K_x\colon X\ra \cF_\FF(X;\bH)$, $K_x(y):=K(y,x)$, for all $y\in X$, the previous equality 
can be expressed as
\begin{equation}\label{e:cekawi}
C_K \widehat h=C_K \delta_x h=K_xh,\quad x\in X,\ h\in \cH_x,
\end{equation}
consequently,
\begin{equation}\label{e:racek}
\ran(C_K)=\lin \{K_x\mid x\in X\}.
\end{equation}
With notation as in \eqref{e:ipzero} 
and \eqref{e:ipk}, we have
\begin{equation}\label{e:convzero} \langle C_Kf,g\rangle_0=\langle f,g\rangle_K,
\quad f,g\in \cF_\FF^0(X;\bH).
\end{equation}

By definition, the kernel $K$ is \emph{positive semidefinite} if it is Hermitian and the sesquilinear form 
$\langle\cdot,\cdot\rangle_K$ 
is nonnegative, that is, if $\langle f,f\rangle_K\geq 0$ for all 
$f\in\cF_\FF^0(X;\bH)$, equivalently, if for all $n\in \NN$,
all $x_1,\ldots, x_n\in X$, and all $h_1\in \cH_{x_1},\ldots, h_n\in \cH_{x_n}$, 
we have
\begin{equation}\label{e:psk} \sum_{i,j=1}^n \langle K(x_j,x_i)h_i,h_j
\rangle_{\cH_{x_j}} \geq 0.
\end{equation}
Another equivalent way of expressing \eqref{e:psk} is to say that the operator block 
matrix $[K(x_j,x_i)]_{i,j=1}^n$, when 
viewed as a bounded linear operator acting in the orthogonal 
direct Hilbert sum $\cH_{x_1}\oplus \cdots\oplus 
\cH_{x_n}$, is a nonnegative operator. On the other hand, the kernel $K$ is 
positive semidefinite if and only if the convolution operator, as defined in 
\eqref{e:conv}, is nonnegative on the inner product space 
$(\cF_\FF^0(X;\bH);\langle\cdot,\cdot\rangle_0)$, more precisely,
\begin{equation}
\langle C_Kf,f\rangle_0\geq 0,\quad f\in\cF_\FF^0(X;\bH).
\end{equation}

In case $\FF=\CC$, the fact that inequality \eqref{e:psk} holds for $n=2$ implies that $K$ is Hermitian, but the statement is not 
true for $\FF=\RR$, in view of the existence of $2\times 2$ matrices, for example, $\begin{bmatrix} 1 & 1 \\ 0 & 1\end{bmatrix}$,
which are positive but not Hermitian.
Because of that, in that case, a positive semidefinite kernel will always be 
assumed to be Hermitian. The collection of all
positive semidefinite $\bH$-operator valued kernels on $X$ is denoted by $\fK^+(X;\bH)$ and it is easy
to see that $\fK^+(X;\bH)$ is a strict convex cone of the $*$-vector space $\fK(X;\bH)$. Then, one considers the order $\leq$ 
induced by the cone $\fK^+(X;\bH)$ on the real vector space $\fK^h(X;\bH)$, that is, for $K,L\in \fK^h(X;\bH)$, we say that $L\leq K$ 
if $K-L\in \fK^+(X;\bH)$. With this notation, if $K\in \fK^h(X;\bH)$, then $K\in \fK^+(X;\bH)$ is equivalent with $K\geq 0$.

\subsection{Hilbert Space Linearisations.}\label{ss:hsl} Given an arbitrary bundle of Hilbert 
spaces $X\ast \bH$ and an $\bH$-operator valued kernel $K$, a \emph{Hilbert space
linearisation}, or a 
\emph{Hilbert space Kolmogorov decomposition}, of $K$ is, by definition, a pair $(\cK;V)$ subject to
the following conditions.
\begin{itemize}
\item[(kd1)] $\cK$ is a Hilbert space over $\FF$.
\item[(kd2)] $V=\{V(x)\}_{x\in X}$ is an 
operator bundle such that $V(x)\in \cB(\cH_x,\cK)$ for all $x\in X$.
\item[(kd3)] $K(x,y)=V(x)^* V(y)$ for all $x,y\in X$.
\end{itemize}
A linearisation $(\cK;V)$ is called \emph{minimal} if
\begin{itemize}
\item[(kd4)] $\cK$ is the closed span of $\{V(x)\cH_x\mid x\in X\}$.
\end{itemize}

The following theorem is Theorem 2 in \cite{GheondeaTilki} and 
a generalised version of the classical result  of A.N.~Kolmogorov \cite{Kolmogorov} 
and its many other partial generalisations, such
as K.R.~Parthasarathy and K.~Schmidt \cite{ParthasarathySchmidt}, D.J.~Evans and J.T.~Lewis 
\cite{EvansLewis}. This is also a special case of Theorem~3.1 from T.~Constantinescu and A.~Gheondea 
\cite{ConstantinescuGheondea1}. 

\begin{theorem}\label{t:kolmogorov} 
Given an arbitrary bundle of Hilbert 
spaces $X\ast \bH$ and an $\bH$-operator valued kernel $K$, the 
following assertions are equivalent.
\begin{itemize}
\item[(a)] $K$ is positive semidefinite.
\item[(b)] $K$ has a Hilbert space linearisation.
\end{itemize}
In addition, if $K$ is positive semidefinite then a minimal Hilbert space linearisation 
$(\cK;V)$ exists and it is unique, modulo unitary equivalence, that 
is, for any other minimal Hilbert space linearisation $(\cK';V')$ of $K$ there exists
a unitary operator $U\colon \cK'\ra \cK$ such that $V(x)=UV'(x)$, for all $x\in X$.
\end{theorem}

It is useful to recall the construction of the minimal linearisation $(\cK;V)$, provided that the kernel $K$ is supposed positive 
semidefinite, following \cite{GheondeaTilki}.  
First, we recall a formalisation of the quotient completion to a Hilbert space of an $\FF$-vector space 
$\cV$ with respect to a given nonnegative sesquilinear form $\cV\times\cV\ni (u,v)\mapsto q(u,v)\in\FF$, as 
follows, which is classical, see \cite{ConstantinescuGheondea1}, \cite{GheondeaTilki}. A pair $(\cH;\Pi)$ is called a \emph{Hilbert space induced} by $(\cV;q)$ if:
\begin{itemize}
\item[(ihs1)] $\cH$ is a Hilbert space.
\item[(ihs2)] $\Pi\colon \cV\ra \cH$ is a linear operator with dense range.
\item[(ihs3)] $q(u,v)=\langle \Pi u,\Pi v\rangle_\cH$, for all $u,v\in\cV$.
\end{itemize} 
Such an induced Hilbert space always exists and is unique, up to a unitary operator. More precisely, we
will use the following construction. Consider the vector subspace of $\cV$ defined by 
\begin{equation}\label{e:nequ}
\cN_q:=\{u\in\cV\mid q(u,u)=0\}=\{u\in\cV\mid q(u,v)=0\mbox{ for all }v\in\cV\},
\end{equation}
where the equality holds due to the Schwarz Inequality for $q$,
and then consider the quotient vector space $\cV/\cN_q$. Letting
\begin{equation}
\widetilde q(u+\cN_q,v+\cN_q):= q(u,v),\quad u,v\in \cV,
\end{equation}
we have a pre-Hilbert space $(\cV/\cN_q;\widetilde q)$ that can be completed to a Hilbert space 
$(\cH_q;\langle\cdot,\cdot\rangle_\cH)$. Letting $\Pi_q\colon \cV\ra \cH_q$ be defined by
\begin{equation}
\Pi_q u:=u+\cN_q\in \cV/\cN_q \subseteq \cH_q,\quad u\in \cV,
\end{equation} 
it is easy to see that $(\cH_q,\Pi_q)$ is a Hilbert space induced by $(\cV;q)$.

We will need the construction of a certain minimal Hilbert space linearisation.

\begin{proof}[Proof of the implication (a)$\Ra$(b) in Theorem~\ref{t:kolmogorov}]
Let us assume the $\bH$-op\-er\-a\-tor valued kernel $K$ is positive semidefinite (and symmetric, in the real case).
We consider the vector space $\cF_\FF^0(X;\bH)$ of $\bH$-valued sections
with finite support and the positive semidefinite, Hermitian, sesquilinear form
$\langle \cdot,\cdot\rangle_K$ defined as in \eqref{e:ipk}. We consider 
\begin{align}\label{e:cenek}
\cN_K & =\{f\in\cF_\FF^0(X;\bH)\mid \langle f,f\rangle_K=0\} \\
& =\{f\in\cF_\FF^0(X;\bH)\mid \langle f,g\rangle_K=0
\mbox{ for all }g\in F_0(X;\bH)\},\nonumber \end{align} then
let $(\cH_K;\Pi_K)$ be the Hilbert space induced by $(\cF_\FF^0(X;\bH);\langle\cdot,\cdot\rangle_K)$,
and let $\cK:=\cH_K$.
For each $x\in X$ let $V(x)\colon \cH_x\ra\cK$ be the operator defined by
\begin{equation}\label{e:vexeh}
V(x)h=\Pi_{K}(\widehat h)= \widehat h+\cN_K=\delta_x h+\cN_K\in\cK,\quad h\in\cH_x,\end{equation}
with notation as in \eqref{e:deltax} and \eqref{e:delta}. Then, for all $x,y\in X$ and all $h\in\cH_x$ and 
$g\in\cH_y$, in view of \eqref{e:cekaw} we have
\begin{align*} \langle V(x)h,V(y)g\rangle_K &  =\langle  \widehat h+\cN,\widehat g+\cN \rangle_K 
= \langle C_K \widehat h,\widehat g\rangle_0\\
& = \langle K_x h,\widehat g\rangle_0
 = \langle K(y,x)h,g\rangle_{\cH_y},
\end{align*} hence
$K(y,x)=V(y)^*V(x)$ for all $x,y\in X$. 

Then, for all $x\in X$ and all $h\in\cH_x$, we have
\begin{equation*} \langle V(x) h,V(x)h\rangle_K=
\langle K(x,x)h,h\rangle_{\cH_x}\leq \|K(x,x)\| \|h\|_{\cH_x}^2,
\end{equation*} hence $V(x)$ is bounded for all $x\in X$. 
Note that, by the very definition, $\cK$ is the closed span of $\{V(x)\cH_x\mid x\in X\}$.
All these show that $(\cK;V)$ is a minimal Hilbert space linearisation of the 
$\bH$-kernel $K$.
\end{proof}

Due to the uniqueness part in the previous theorem, for any $K\in\fK^+(X;\bH)$, we 
denote by $(\cK_K;V_K)$
the minimal Hilbert space linearisation of $K$, as constructed during the proof of the 
implication (a)$\Ra$(b).

\begin{remark} Let $(\cK;V)$ be a Hilbert space linearisation of the positive semidefinite 
$\bH$-kernel $K$, with notation as before. If $(\cK;V)$ is not minimal, one can consider 
the Hilbert space
\begin{equation*}
\cK_0:=\Clos\Span\{V(x)\cH_x\mid x\in X\},
\end{equation*}
and observe that, replacing $\cK$ with $\cK_0$ and, for each $x\in X$, replacing 
$V(x)$ with $V_0(x):=P_{\cK_0}V(x)$, we obtain a minimal linearisation $(\cK_0;V_0)$  
of $K$.\end{remark}

\subsection{Reproducing Kernel Hilbert Spaces.}\label{ss:rkhs} Let $X$ be a nonempty set and 
$X\ast \bH$ be a bundle of Hilbert spaces over $\FF$.
Given an $\bH$-operator valued kernel $K$, 
a  \emph{reproducing kernel Hilbert space} associated to $K$ is, by definition, a Hilbert space 
$\cR\subseteq \cF_\FF(X;\bH)$ subject to the following conditions.
\begin{itemize}
\item[(rk1)] $\cR$ is a subspace of $\cF_\FF(X;\bH)$, with all induced algebraic operations.
\item[(rk2)] For all $x\in X$ and $h\in\cH_x$, the section
$K_xh:=K(\cdot,x)h$ belongs to $\cR$. 
\item[(rk3)] For all $f\in\cR$ we have $\langle f(x),h\rangle_{\cH_x}
=\langle f,K_x h\rangle_\cR$, for all $x\in X$ and all $h\in\cH_x$.
\end{itemize}

Property (rk3) is called the \emph{reproducing property}.

\begin{remark}
In view of property (rk2), let us observe that for each $x\in X$, we have a linear operator $K_x\colon \cH_x\ra \cR$. This operator 
is bounded. Indeed, for each $h\in\cH_x$, in view of (rk3) we have
\begin{equation*}
\|K_xh\|_\cR^2= \langle K_xh,K_xh\rangle_\cR=\langle K(x,x)h,h\rangle_{\cH_x}\leq \|K(x,x)\| \|h\|_{\cH_x}^2,
\end{equation*}
where $\|K(x,x)\|$ is the operator norm of $K(x,x)\in\cB(\cH_x)$.
\end{remark}

\begin{proposition}\label{p:rkhs}
Let $\cR$ be a reproducing kernel Hilbert space associated to the $\bH$-operator valued kernel $K$, defined before. Then,
\begin{itemize}
\item[(a)] $K(x,y)=K_x^* K_y$ for all $x,y\in X$, that is, 
$\langle K(x,y)k,h\rangle_{\cH_x}=\langle K_yk,K_xh\rangle_\cR$ for all $x,y\in X$, $h\in\cH_x$, $k\in\cH_y$.
\item[(b)] $K$ is Hermitian/symmetric and positive semidefinite.
\item[(c)] $\{K_xh\mid x\in X,\ h\in\cH_x\}$ is total in $\cR$.
\item[(d)] $K$ is uniquely associated to $\cR$.
\end{itemize}
\end{proposition}

\begin{proof} (a) This is straightforward from the reproducing property and the 
definition of $K_x$, more precisely, for all $x,y\in X$, $h\in\cH_x$, and $k\in\cH_y$, we have
\begin{equation*}\langle K_yk,K_xh\rangle_\cR=\langle K_y(x)k,h\rangle_{\cH_x}=\langle K(x,y)k,h\rangle_{\cH_x}.
\end{equation*}

(b) First, by (a), for all $x,y\in X$, $h\in\cH_x$, and $k\in\cH_y$, we have
\begin{align*}
\langle K(x,y)k,h\rangle_{\cH_x} & =\langle K_yk,K_xh\rangle_\cR=\overline{\langle K_xh,K_yk\rangle_\cR}=\overline{\langle 
K(y,x)h,k\rangle_{\cH_y}}\\
& =\overline{\langle h,K(y,x)^*k\rangle_{\cH_x}}=\langle K(y,x)^*k,h\rangle_{\cH_x}.
\end{align*}
Then, for arbitrary $n\in\NN$, $x_1,\ldots,x_n\in X$ and $h_i\in\cH_{x_i}$, for $i=1,\ldots,n$, we have
\begin{align*}
\sum_{i,j=1}^n \langle K(x_i,x_j)h_j,h_i\rangle_{\cH_{x_i}} & = \sum_{i,j=1}^n \langle K_{x_j}h_j,K_{x_i}h_i\rangle_\cR 
= \langle \sum_{j=1}^n K_{x_j}h_j,\sum_{i=1}^n  K_{x_i}h_i\rangle_\cR \\
& = \| \sum_{j=1}^n K_{x_j}h_j\|_\cR^2\geq 0.
\end{align*}

(c) Let $f\in \cR$ be such that $f\perp K_xh$, for all $x\in X$ and $h\in\cH_x$. 
By the reproducing property, for all $x\in X$ and $h\in\cH_x$, we have
$f(x)h=\langle f,K_xh\rangle_\cR=0$, hence $f=0$.

(d) Let $K$ and $L$ be two kernels associated to the same reproducing kernel Hilbert space $\cR$. 
Then, for all $x,y\in X$, $h\in\cH_x$, $k\in\cH_y$, we have
\begin{align*}
\langle K(x,y)k,h\rangle_{\cH_x} & =\langle K_y(x)k,h\rangle_{\cH_x}=\langle K_yk,L_xh\rangle_\cR=\ol{\langle L_xh,K_yk\rangle_\cR}\\ 
& =\ol{\langle L(y,x)h,k\rangle_{\cH_y}}=\ol{\langle h,L(y,x)^*k\rangle_{\cH_x}} \\ & 
= \langle L(y,x)^*k,h\rangle_{\cH_x}=\langle L(x,y)k,h\rangle_{\cH_x},
\end{align*}
hence $K=L$.
\end{proof}

One of the most important properties of a reproducing kernel Hilbert space consists in the fact that, as a 
function space, its topology makes continuous all evaluation operators. Recall that $\cF_\FF(X;\bH)$ is a complete Hausdorff 
locally convex space when endowed with the seminorms $p_x$, $x\in X$, defined as in \eqref{e:pexef}. The next result
is Theorem~4 in \cite{GheondeaTilki} and a generalisation of classical results, cf.\ \cite{Aronszajn}.

\begin{theorem}\label{t:evop} With notation as before, let $\cH$ be a Hilbert space in the vector space $\cF_\FF(X;\bH)$.
The following assertions are equivalent.
\begin{itemize}
\item[(a)] $\cH$ is a reproducing kernel space of $\bH$-valued maps on $X$.
\item[(b)] For any $x\in X$, the linear operator $\cH\ni f\mapsto f(x)\in \cH_x$ is bounded.
\item[(c)] $\cH\hookrightarrow \cF_\FF(X;\bH)$, that is, the Hilbert space $\cH$ is continuously included in 
$\cF_\FF(X;\bH)$.
\end{itemize}
\end{theorem}

The implication (a)$\Ra$(b) in the next theorem is a generalisation of a theorem of E.H.~Moore \cite{Moore} and N.~Aronszajn 
\cite{Aronszajn}, as well as of many other operator valued versions of it, such as 
C.~Carmelli, E.~De Vito, and A.~Toigo \cite{Carmelli}. More general but partially special variants can 
be seen at T.~Constantinescu and A.~Gheondea \cite{ConstantinescuGheondea1} and
S.~Ay and A.~Gheondea \cite{AyGh2015}, \cite{AyGh2017}, \cite{AyGh2019}. This is Theorem~2 
in \cite{GheondeaTilki}.

\begin{theorem}\label{t:rkh} Let $K$ be an $\bH$-operator valued kernel, where $X\ast \bH$ is 
a bundle of Hilbert spaces.

\nr{1}  The 
following assertions are equivalent.
\begin{itemize}
\item[(a)] $K$ is positive semidefinite.
\item[(b)] There exists a reproducing kernel Hilbert space $\cR$ having $K$
its reproducing kernel.
\end{itemize}

\nr{2} The reproducing kernel Hilbert space $\cR$ is uniquely determined by its reproducing kernel $K$.
\end{theorem}

Due to the uniqueness property of the reproducing kernel Hilbert space associated 
to a positive semidefinite $\bH$-operator valued kernel $K$, it is natural to denote
this reproducing kernel Hilbert space by $\cR_K$.

\begin{remark}\label{r:kolrk}\!\! \emph{Minimal Linearisations versus Reproducing 
Kernel Hilbert Spaces}. It is useful to treat more carefully the relation between linearisations and reproducing kernel Hilbert 
spaces, which shows up in the proof of Theorem~\ref{t:rkh}, namely, that there is a 
natural bijective transformation between the unitary equivalence class of minimal
linearisations $(\cK;V)$ of $K$ and the reproducing kernel Hilbert
space $\cR_K$. The transformation from minimal linearisations 
$(\cK;V)$ to the reproducing kernel Hilbert space $\cR_K$ is described during the
proof of the implication (a)$\Ra$(b) of Theorem~\ref{t:rkh}, see Appendix C in \cite{GheondeaTilki}, and it is 
related to the general idea that the concept of reproducing kernel Hilbert space is closely related to that of 
operator range, see P.A.~Fillmore and J.P.~Williams 
\cite{FillmoreWilliams}, for a survey.  More precisely, one defines
\begin{equation}\label{e:rel}
\cR:=\{V(\cdot)^*f\mid f\in\cK\},
\end{equation}
and note that $\cR$ is a vector subspace of $\cF_\FF(X;\bH)$ on which the inner product 
$\langle\cdot,\cdot\rangle_\cR$ defined by
\begin{equation}\label{e:relip}
\langle V(\cdot)^*f,V(\cdot)^*g\rangle_\cR:=\langle f,g\rangle_\cK,\quad f,g\in \cK,
\end{equation}
turns $\cR$ into an $\bH$-valued reproducing kernel Hilbert space with $\bH$-operator valued reproducing kernel $K$.

In the following we
describe the inverse of this transformation.
Let $(\cR;\langle\cdot,\cdot\rangle_\cR)$ be a 
reproducing kernel Hilbert space with reproducing kernel $K$. 
We define the operator bundle $V=\{V(x)\}_{x\in X}$ by
\begin{equation}\label{e:v} V(x)h:=K_xh,\quad x\in X,\ h\in\cH_x,
\end{equation} and remark that $V(x)\colon\cH_x\ra\cR$ for all $x\in X$. 
By means of the reproducing property (rk3) of the kernel $K$, for all $x,y\in X$, $h\in\cH_x$, and $k\in\cH_y$, we have
\begin{equation*} \langle V(x)h,V(x)h\rangle_\cR=\langle K_xh,K_xh\rangle_\cR
=\langle K(x,x)h,h\rangle_{\cH_x}\leq \|K(x,x)\| \|h\|^2_{\cH_x},
\end{equation*} hence $V(x)\in\cB(\cH_x,\cR)$. Also, using once more
the reproducing property (rk3) of $K$, it follows that, for all $x,y\in X$, $h\in\cH_x$,
and $g\in\cH_y$, we have
\begin{equation*} \langle V(y)^*V(x)h,g\rangle_{\cH_y}=\langle V(x)h,
V(y)g\rangle_\cR=\langle K_xh,K_yg\rangle_\cR=\langle K(y,x)h,g\rangle_{\cH_y}.
\end{equation*}
Therefore, $K(y,x)=V(y)^*V(x)$ for all $x,y\in X$ and hence, $(\cR;V)$ is a 
linearisation of $K$. In addition, using the minimality property,
it is easy to see that the linearisation $(\cR;V)$ is minimal as well.
\end{remark}

\section{Positive Semidefinite Kernels Invariant under Actions of $*$-Semigroupoids}\label{s:psk}

\subsection{$*$-Semigroupoids}\label{ss:ssg}
We first review definitions and a few facts concerning $*$-semigroupoids, following 
\cite{GheondeaUdrea}.

 A \emph{semigroupoid} \cite{Tilson} is a quintuple $(\Gamma;S;d;c;\cdot)$ subject to the 
following conditions.
\begin{itemize}
\item[(SG1)] $\Gamma$ and $S$ are nonempty sets.
\item[(SG2)] $d\colon \Gamma\ra S$ and $c\colon \Gamma\ra S$ are maps called, respectively, the \emph{domain map} and the \emph{codomain map}. 
\item[(SG3)] For every $\alpha,\beta\in\Gamma$ such that $d(\alpha)=c(\beta)$, there exists a unique
element $\alpha\cdot\beta\in\Gamma$, 
with $d(\beta)=d(\alpha\cdot\beta)$ and $c(\alpha\cdot\beta)=c(\alpha)$.
\item[(SG4)] For any $\alpha,\beta,\gamma\in\Gamma$ such that $d(\beta)=c(\gamma)$ 
and $d(\alpha)=c(\beta)$, we have 
\begin{equation*}\alpha\cdot (\beta\cdot\gamma)=(\alpha\cdot\beta)\cdot\gamma.\end{equation*}
\end{itemize}

As usual, we will simply say that $\Gamma$ is a semigroupoid and, by this, we mean that 
there exist the set $S_\Gamma$, called the \emph{set of symbols} of $\Gamma$, 
the maps $d_\Gamma$ and $c_\Gamma$, and the operation $\cdot$ 
such that the quintuple $(\Gamma;S_\Gamma;d_\Gamma;c_\Gamma;\cdot)$ 
satisfies the properties (SG1) through (SG4). Also, when there is no danger of confusion, 
we will drop the lower index $\Gamma$. We will also write $\alpha\beta$ instead of $\alpha\cdot\beta$, 
whenever the operation is possible. In this respect, one denotes
\begin{equation*}
\Gamma^{(2)}:=\{(\alpha,\beta)\mid \alpha,\beta\in\Gamma,\ c(\beta)=d(\alpha)\},
\end{equation*}
and call it the \emph{set of composable pairs} in $\Gamma$. 

For each $s,t\in S$ we consider the fibres
\begin{equation*}
\Gamma^t:=\{\alpha\in \Gamma\mid c(\alpha)=t\},\quad \Gamma_s:=\{\alpha\in \Gamma
\mid d(\alpha)=s\},\quad \Gamma_s^t:=\Gamma_s\cap \Gamma^t.
\end{equation*}
In general, any of these fibres can be empty. Letting 
$S_0:=\{ s\in S\mid \Gamma^s=\emptyset=\Gamma_s\}$, the set
of isolated points in $S$, by replacing $S$ with $S\setminus S_0$ we remove those points in $S$ that are 
neither the domain nor the codomain of any element $\gamma\in \Gamma$. Without loss of generality, we can 
thus assume that, for any $s\in S$, either $\Gamma_s$ or $\Gamma^s$ is nonempty.

If $\Gamma_s\neq\emptyset$ for all $s\in S$ then 
$\{\Gamma_s\mid s\in S\}$ is a partition of $\Gamma$ and hence it defines an 
equivalence relation on $\Gamma$: $\alpha\sim_d\beta$ if $d(\alpha)=d(\beta)$. 
Similarly,  if $\Gamma^s\neq \emptyset$ for all $s\in S$ then
$\{\Gamma^s\mid s\in S\}$ is a partition of $\Gamma$ and hence it defines an
equivalence relation on $\Gamma$: $\alpha\sim_c\beta$ if $c(\alpha)=c(\beta)$.
 
 Also, $\Gamma_s^t$ may be empty for some $s,t\in S$. 
 If $\Gamma_s^t$ is not empty for each $(s,t)\in S\times S$,
 equivalently, the map $(d,c)\colon \Gamma\ra S\times S$ is surjective,
 the semigroupoid is called \emph{transitive} and, in this case, $\{\Gamma_s^t\mid s,t\in S\}$ 
 is a partition of $\Gamma$ which defines the equivalence relation: $\alpha\sim_{d,c}\beta$ if 
 $d(\alpha)=d(\beta)$ and $c(\alpha)=c(\beta)$.
 If $s\in S$ and $\Gamma_s^s$ is not empty then $\Gamma_s^s$ is a semigroup, 
 called the \emph{isotropy semigroup} at $s$.
 
 A semigroupoid $\Gamma$ has a \emph{unit} if there exists an injective map 
$\epsilon \colon S_\Gamma\ra \Gamma$, for which we use the notation $\epsilon=\{\epsilon_s\}_{s\in S_\Gamma}$, 
subject to the following conditions.
\begin{itemize}
\item[(U1)] For every $s\in S_\Gamma$, $d_\Gamma(\epsilon_s)=c_\Gamma(\epsilon_s)=s$.
\item[(U2)] For any $s\in S_\Gamma$ and any $\alpha\in\Gamma^s$, we have 
$\epsilon_s\alpha=\alpha$.
\item[(U3)] For any $s\in S_\Gamma$ and any $\alpha\in \Gamma_s$, we have 
$\alpha \epsilon_s=\alpha$.
\end{itemize}

A unit $\epsilon$ of the semigroupoid $\Gamma$, if it exists, is unique. In addition,
since the unit $\epsilon$ is an injective map this yields an embedding of $S_\Gamma$ in $\Gamma$.
In this case, one usually denotes $\Gamma^0:=\{\epsilon_s\mid s\in S_\Gamma\}$, 
and calls it the \emph{set of units} of $\Gamma$. Clearly, $\Gamma^0$ is in a bijective correspondence
with $S_\Gamma$ and for semigroupoids with unit one usually identifies them.
If the semigroupoid $\Gamma$ has a unit then, for each $s\in S_\Gamma$ 
the isotropy semigroup $\Gamma_s^s$ has a unit. A semigroupoid with unit can be defined, equivalently, as a 
small category, that is, a category in which the classes of objects and arrows are sets.

Given a semigroupoid $\Gamma$, an \emph{involution} on $\Gamma$ is a map 
$\Gamma\ni \alpha\mapsto \alpha^*\in\Gamma$ subject to the following conditions.
\begin{itemize}
\item[(I1)] For any $\alpha\in\Gamma$, we have $d(\alpha^*)=c(\alpha)$ and $c(\alpha^*)=d(\alpha)$.
\item[(I2)] For any $(\alpha,\beta)\in \Gamma^{(2)}$, we have $(\alpha\beta)^*=\beta^*\alpha^*$.
\item[(I3)] For any $\alpha\in\Gamma$, we have $(\alpha^*)^*=\alpha$.
\end{itemize}
A semigroupoid with a specified involution $*$ will be called a \emph{$*$-semigroupoid}.

Note that, by property (I3), involutions are bijective maps.  If the
$*$-semigroupoid $\Gamma$ has a unit $\epsilon$, then 
\begin{equation*}
\epsilon_s^*=\epsilon_s,\quad s\in S_\Gamma.
\end{equation*}
For any $s\in S_\Gamma$ 
such that $\Gamma_s^s\neq\emptyset$, the isotropy semigroup $\Gamma_s^s$ has an involution.

Let $\Gamma$ be a semigroupoid and $X$ a nonempty set. A \emph{left action} of
$\Gamma$ on $X$ is a pair $(a;\cdot)$ subject to the following conditions.
\begin{itemize}
\item[(A1)] $a\colon X\ra S_\Gamma$ is a surjective map, called \emph{anchor}.
\item[(A2)] For any $x\in X$ and any $\alpha\in\Gamma_{a(x)}$,
there exists a unique element $\alpha\cdot x\in X$ such that $a(\alpha\cdot x)=c(\alpha)$.
\item[(A3)] For any $x\in X$ and any $(\alpha,\beta)\in\Gamma^{(2)}$ such that 
$\beta\in\Gamma_{a(x)}$, we have
\begin{equation*}
(\alpha\beta)\cdot x=\alpha\cdot(\beta\cdot x).
\end{equation*}
\end{itemize}
If, in addition, $\Gamma$ has a unit $\epsilon$, then the left action $(a;\cdot)$ is called \emph{unital} if
\begin{equation*}
\epsilon_{a(x)}\cdot x=x,\quad x\in X.
\end{equation*}

\begin{example}\label{ex:actit} (a) Any semigroupoid $\Gamma$ has a natural left action on itself. To see this, 
let $a\colon \Gamma\ra S_\Gamma$ be defined by $a(\beta):=c(\beta)$. In addition,
for any $\alpha\in\Gamma$ and any $\beta\in \Gamma_{a(\alpha)}$, the action is defined by 
$\beta\cdot \alpha:=\beta \alpha$.

(b) Any semigroupoid $\Gamma$ has a natural left action on $S_\Gamma$. Here the anchor map 
$a\colon S_\Gamma\ra S_\Gamma$ is the identity map,  $a(s)=s$ for all $s\in S_\Gamma$, and for any 
$s\in S_\Gamma$ and any $\alpha\in \Gamma_s$, the action is $\alpha\cdot s=c(\alpha)$.
\end{example}

In the following we recall three basic examples that will play a major role in this article.

\begin{example}\label{ex:segah} (\emph{$*$-Semigroupoids Modelled by Bounded Operators on Hilbert Spaces})
Let $S$ be a nonempty set and $\bH=\{\cH_s\}_{s\in S}$ a \emph{bundle of Hilbert spaces}
over $\FF$ (either $\RR$ or $\CC$). Let 
\begin{equation}\label{e:gamabeh}
\Gamma_{\bH}=\bigsqcup_{s,t\in S} \cB(\cH_s,\cH_t),
\end{equation}
where $\cB(\cH_s,\cH_t)$ denotes the vector space of all bounded linear operators 
$T\colon \cH_s\ra\cH_t$.

For every $T\in \Gamma_{\bH}$, there exists uniquely $s,t\in S$ such that 
$T\in\cB(\cH_s,\cH_t)$ and then
we define $d(T)=s$ and $c(T)=t$. In this way, $d\colon\Gamma_{\bH}\ra S$ and 
$c\colon \Gamma_{\bH}\ra S$ are surjective maps.

The operation of composition in $\Gamma_{\bH}$ 
is defined by operator composition: if $A,B\in\Gamma_{\bH}$
are such that $d(A)=c(B)=u\in S$, then $A\in\cB(\cH_u,\cH_t)$ for some $t\in S$ and 
$B\in\cB(\cH_s,\cH_u)$ for some $s\in S$. Then $AB\in \cB(\cH_s,\cH_t)$
is the composition of the 
operators $A$ and $B$, in this order. Thus, $(\Gamma_{\bH},S,d_{\bH},c_{\bH},\cdot)$ is a semigroupoid.

The semigroupoid $\Gamma_{\bH}$ has the unit $\epsilon\colon S\ra\Gamma_{\bH}$ defined by 
$\epsilon_s:=I_{\cH_s}$, where $I_{\cH_s}$ denotes the identity operator on $\cH_s$, for all $s\in S$.
It also has a natural involution: $\Gamma_{\bH}\ni T\mapsto T^*\in\Gamma_{\bH}$, where, 
if $T\in\cB(\cH_s,\cH_t)$, by $T^*\in\cB(\cH_t,\cH_s)$ we denote the Hilbert space adjoint operator.

In addition, the semigroupoid $\Gamma_{\bH}$ has a natural left action on $S$ defined by 
$T\cdot s:=t$, where $T\in\cB(\cH_s,\cH_t)$.
\end{example}

\begin{example} (\emph{Semigroupoids Modelled on Vector Spaces}) 
\label{ex:semes} Let $S$ be a nonnempty 
set and $\bD=\{\cD_s\}_{s\in S}$ be a bundle of vector spaces over $\FF$. We let
\begin{equation*}
\Gamma_{\bD}=\bigsqcup_{s,t\in S}\cL(\cD_s,\cD_t),
\end{equation*}
where by $\cL(\cD_s,\cD_t)$ we understand the vector space of all linear operators 
$T\colon \cD_s\ra\cD_t$. As in the previous example, for any $T\in \Gamma_{\bD}$ there exists
uniquely $s,t\in S$ such that $T\in\cL(\cD_s,\cD_t)$ and we let $d(T):=s$ and $c(T):=t$, hence
$d\colon \Gamma_{\bD}\ra S$ and $c\colon \Gamma_{\bD}\ra S$ are surjections. The partial
composition is defined as in the previous example and, in this way, $\Gamma_{\bD}$ becomes
a semigroupoid with unit $\epsilon\colon S\ra\Gamma_{\bD}$, $\epsilon(s):=I_{\cD_s}$. As in the 
previous example, there is a natural left action of $\Gamma_{\bD}$ onto $S$.
\end{example}

\begin{example}\label{ex:semuh}
 (\emph{$*$-Semigroupoids Modelled by Unbounded Operators in Hilbert Spaces.})
\label{ex:semeseh} 
Let us assume, with notation as in the previous example,
that for each $s\in S$ the vector space 
$\cD_s$ is a dense subspace of a Hilbert space $\cH_s$ and let
\begin{equation}\label{e:gabed}
\Gamma_{\bH;\bD}:=\bigsqcup_{s,t\in S} \cL^*(\cD_s,\cD_t),
\end{equation}
where, for each $s,t\in S$, we let $\cL^*(\cD_s,\cD_t)$ denote the vector space of all
linear operators $T\colon \cD_s(\subseteq \cH_s)\ra\cH_t$ subject to the following assumptions.
\begin{itemize}
\item[(i)] $T\cD_s\subseteq \cD_t$.
\item[(ii)] $ \cD_t\subseteq \dom(T^*)$ and $T^*\cD_t\subseteq\cD_s$.
\end{itemize}
Here, by $T^*\colon\dom(T^*)(\subseteq \cH_t)\ra\cH_s$ we understand 
the adjoint operator (possibly unbounded) defined in the usual sense,
\begin{equation*}
\dom(T^*):=\{k\in \cH_t\mid \cD_s\ni h\mapsto \langle Th,k\rangle_{\cH_t}\mbox{ is bounded}\},
\end{equation*}
and
\begin{equation*}
\langle Th,k\rangle_{\cH_t}=\langle h,T^*k\rangle_{\cH_s},\quad h\in\cD_s,\quad k\in\dom(T^*).
\end{equation*}
Note that, in this way, any operator $T$ in $\cL^*(\cD_s,\cD_t)$ is closable. The involution on $\Gamma_{\bH,\bD}$ is
defined by $\cL^*(\cD_s,\cD_t)\ni T\mapsto T^*|_{\cD_t}$,  for all $s,t\in S$.

Then, it is easy to see that $\Gamma_{\bH;\bD}\ni T\mapsto T^*\in\Gamma_{\bH;\bD}$ is an involution
which makes $\Gamma_{\bH;\bD}$ a $*$-semigroupoid with unit. As in the previous example, 
there is a natural left action of $\Gamma_{\bH;\bD}$ onto $S$.
\end{example}

A semigroupoid $\Gamma$ is called a \emph{groupoid} \cite{Brandt} if it has a unit $\epsilon$ and there exists a 
map $\Gamma\ni \alpha\mapsto \alpha^{-1}\in \Gamma$ such that, for every $\alpha\in \Gamma$,
we have $d(\alpha^{-1})=c(\alpha)$, $c(\alpha^{-1})=d(\alpha)$, and
\begin{equation}
\alpha\alpha^{-1}=\epsilon_{c(\alpha)},\quad \alpha^{-1}\alpha=\epsilon_{d(\alpha)}.
\end{equation}

If $\Gamma$ is a groupoid then its unit set 
$\Gamma^0=\{\alpha^{-1}\alpha\mid \alpha\in\Gamma\}$ is naturally identified with the set 
$S_\Gamma$, and there is no need to specify $S_\Gamma$ beforehand, which is done by 
some authors. 

A semigroupoid $\Gamma$ is called an \emph{inverse semigroupoid}, e.g.\ see \cite{Liu}, if for any 
$\alpha\in\Gamma$ there exists a unique $\alpha^\prime\in\Gamma$ such that 
$\alpha\alpha^\prime\alpha=\alpha$ 
and $\alpha^\prime\alpha\alpha^\prime=\alpha^\prime$. Note that, in particular, this means that
$d_\Gamma(\alpha^\prime)=c_\Gamma(\alpha)$ and 
$c_\Gamma(\alpha^\prime)=d_\Gamma(\alpha)$, for all $\alpha\in \Gamma$.

If $\Gamma$ is an inverse semigroupoid then, for any $\alpha\in\Gamma$, we have
$(\alpha^\prime)^\prime=\alpha$ and for any $(\alpha,\beta)\in\Gamma^{(2)}$ we have 
$(\alpha\beta)^\prime=\beta^\prime\alpha^\prime$. In particular, any inverse semigroupoid is a 
$*$-semigroupoid with $\alpha^*:=\alpha^\prime$. Also, any groupoid $\Gamma$ is an inverse semigroupoid 
with $\gamma^\prime:=\gamma^{-1}$.

\subsection{Representations of $*$-Semigroupoids}\label{ss:rssg}
 Here we review basic definitions and facts concerning representations of 
$*$-semigroupoids, following \cite{GheondeaUdrea}.

Let $\Gamma$ and $\Lambda$ be two semigroupoids. A pair $(\phi;\Phi)$ is 
a \emph{semigroupoid morphism} from $\Gamma$ to $\Lambda$ if the following conditions hold.
\begin{itemize}
\item[(SM1)] $\phi\colon S_\Gamma\ra S_\Lambda$ and $\Phi\colon \Gamma\ra \Lambda$ are maps.
\item[(SM2)] For any $\alpha\in\Gamma$ we have $d_\Lambda(\Phi(\alpha))=\phi(d_\Gamma(\alpha))$ 
and $c_\Lambda(\Phi(\alpha))=\phi(c_\Gamma(\alpha))$.
\item[(SM3)] For any $(\alpha,\beta)\in \Gamma^{(2)}$ we have 
$(\Phi(\alpha),\Phi(\beta))\in\Lambda^{(2)}$ and
\begin{equation*}
\Phi(\alpha\beta)=\Phi(\alpha)\Phi(\beta).
\end{equation*}
\end{itemize}
The map $\phi$ is called the \emph{aggregation map}.

On the other hand, if both $\Gamma$ and $\Lambda$ are $*$-semigroupoids, the semigroupoid
morphism $(\phi;\Phi)$ is called a \emph{$*$-morphism} if
\begin{equation*}
\Phi(\alpha^*)=\Phi(\alpha)^*,\quad \alpha\in\Gamma.
\end{equation*}

The pair $(\phi;\Phi)$ is called a \emph{monomorphism}, \emph{epimorphism}, \emph{isomorphism}, 
of semigroupoids if both $\phi$ and $\Phi$ are injective, surjective, and bijective, respectively.

In this article we will use two different types of representations of $*$-semigroupoids, the first one with bounded 
operators and the second one with unbounded operators.

 \begin{definition}\label{d:serep}
Let $\Gamma$ be a semigroupoid. A \emph{representation} of $\Gamma$ on a
bundle of Hilbert spaces $\bH=\{\cH_s\}_{s\in S}$ is a semigroupoid morphism 
$(\phi;\Phi)$ of $\Gamma$ to the semigroupoid $\Gamma_{\bH}$, see Example~\ref{ex:segah}.
More precisely:
\begin{itemize}
\item[(R1)] $\phi\colon S_\Gamma\ra S$ and $\Phi\colon \Gamma\ra\Gamma_{\bH}$ 
are maps.
\item[(R2)] For any $\alpha\in\Gamma$, we have 
$\Phi(\alpha)\in\cB(\cH_{\phi(d(\alpha))},\cH_{\phi(c(\alpha))})$. 
\item[(R3)] For any 
$(\alpha,\beta)\in\Gamma^{(2)}$, we have $\Phi(\alpha\beta)=\Phi(\alpha) \Phi(\beta)$.
\end{itemize}

If $\Gamma$ is a $*$-semigroupoid then the morphism $(\phi;\Phi)$ is called a 
\emph{$*$-represen\-tation} of $\Gamma$ on a bundle of Hilbert spaces
$\bH$ if it is a $*$-morphism of $*$-semigroupoids from $\Gamma$ to $\Gamma_{\bH}$.
\end{definition}

The more general concept of representation of a $*$-semigroupoid on 
$*$-semigroupoids of type $\Gamma_{\bH,\bD}$, see Example~\ref{ex:semeseh}, provides representations with 
unbounded operators. 

\begin{definition}\label{d:urep} Let $\Gamma$ be a semigroupoid and consider two bundles
of vector spaces $\bD=\{\cD_s\}_{s\in S}$ and $\bH=\{\cH_s\}_{s\in S}$, where $\cD_s$ is a dense
subspace in the Hilbert space $\cH_s$, for all $s\in S$. With notation as in 
Example~\ref{ex:semuh}, a \emph{generalised representation} of $\Gamma$ on the pair of
bundles $(\bD;\bH)$ is a pair of maps $(\phi;\Phi)$ subject to the following conditions.
\begin{itemize}
\item[(UR1)] $\phi\colon S_\Gamma\ra S$ and $\Phi\colon \Gamma\ra\Gamma_{\bH,\bD}$ are maps.
\item[(UR2)] For any $\alpha\in \Gamma$ we have $\Phi(\alpha)$ a linear operator such that
$\cD_{\phi(d(\alpha))}\subseteq \dom(\Phi(\alpha))$ and 
$\Phi(\alpha)\cD_{\phi(d(\alpha))}\subseteq\cD_{\phi(c(\alpha))}$.
\item[(UR3)] For any $(\alpha,\beta)\in \Gamma^{(2)}$ we have 
$\Phi(\alpha\beta)|_{\cD_{\phi(d(\beta))}}=\Phi(\alpha)\Phi(\beta)|_{\cD_{\phi(d(\beta))}}$.
\end{itemize}

We call the generalised representation $(\phi;\Phi)$ \emph{orthogonal} if the following property holds
\begin{itemize}
\item[(UR4)] For any $\alpha,\beta\in \Gamma$ such that
$\phi(c(\alpha))=\phi(c(\beta))$, but $c(\alpha)\neq c(\beta)$,
 it follows that $\Phi(\alpha)\cD_{\phi(d(\alpha))}\perp 
\Phi(\beta)\cD_{\phi(d(\beta))}$.
\end{itemize}

If, in addition, $\Gamma$ is a $*$-semigroupoid,
then we call $(\phi;\Phi)$ a \emph{generalised $*$-representation} of $\Gamma$ on the bundles $(\bD,\bH)$ if
the following conditions hold.
\begin{itemize}
\item[(UR5)] For any $\alpha\in \Gamma$ we have
$\Phi(\alpha)^*\cD_{\phi(c(\alpha))}\subseteq\cD_{\phi(d(\alpha))}$. 
\item[(UR6)] For any $\alpha\in \Gamma$ we have $\Phi(\alpha^*)|_{\cD_{\phi(c(\alpha))}}
=\Phi(\alpha)^*|_{\cD_{(\phi(c(\alpha))}}$.
\end{itemize}
\end{definition}

In case the representation $(\phi;\Phi)$ is aggregation free, that is,
$\phi$ is injective, it is automatically orthogonal. In general, the fashion in which different pieces of the
representation are aggregated within a bundle of Hilbert spaces makes technical obstructions. 

\subsection{Generalised $\Gamma$-Invariant Hilbert Space Linearisations}\label{ss:ipsdk}
Let $X$ be a nonempty set and $(\Gamma,a)$ a left action of the $*$-semigroupoid $\Gamma$ on $X$, 
with $a\colon X\ra S$ the anchor map. We use notation as in Subsection~\ref{ss:ssg}, 
hence $S$ is the set of symbols of $\Gamma$ and
$d\colon \Gamma\ra S$ and $c\colon \Gamma\ra S$ denote the domain map and the codomain map, 
respectively. Since the anchor map $a$ is surjective, for each $s\in S$ the set
\begin{equation}\label{e:xes} X_s:=\{x\in X\mid a(x)=s\}=a^{-1}(\{s\})
\end{equation}
is nonempty and, consequently, we have a partition induced by the anchor map $a$,
\begin{equation}\label{e:part}
X=\bigcup_{s\in S}X_s,
\end{equation}
which we denote by $P_a:=\{X_s\}_{s\in S}$.

Let $K\colon X\times X\ra\Gamma_{\bH}$ be a kernel, where $\bH=\{\cH_x\mid x\in X\}$ is a bundle of Hilbert 
spaces over the field $\FF$, where the $*$-semigroupoid with unit 
$\Gamma_{\bH}$ is defined in Example~\ref{ex:segah}. 
In particular, $K(x,y)\in \cB(\cH_y,\cH_x)$ for all $x,y\in X$. Recall that such a kernel is called an $\bH$-operator valued kernel 
on $X$. 

An $\bH$-operator valued kernel $K$
is \emph{partially Hermitian} with respect to the partition \eqref{e:part} if,
\begin{equation}\label{e:pah}
K(x,y)^*=K(y,x),\quad s\in S,\ x,y\in X_s.\end{equation} 
The kernel $K$ is called
\emph{partially positive semidefinite} with respect to the partition \eqref{e:part} if, for 
any $s\in S$, for any $n\in\NN$, any $x_1,\ldots,x_n \in X_s$, and any $h_1,\ldots,h_n$ such that 
$h_i\in \cH_{x_i}$ for all $i=1,\ldots,n$, we have 
\begin{equation}\label{e:partpsd}
\sum_{i,j=1}^n \langle K(x_j,x_i)h_i,h_j\rangle_{\cH_{x_j}}\geq 0.
\end{equation}
In case $\FF=\CC$, the property of being partially Hermitian with respect to a partition is a consequence of being partially positive 
semidefinite with respect to that partition but, if $\FF=\RR$, this should be assumed in advance.

With notation as before, 
we assume, in addition, that for all $x\in X$ and all $\alpha\in \Gamma_{a(x)}$ we have 
\begin{equation}\label{e:orbit}\cH_{\alpha\cdot x}=\cH_x.\end{equation} 
This assumption is needed in order to introduce two more concepts that are essential in this paper. First,
for each $\alpha\in \Gamma$ we can consider the linear transformation
\begin{equation}\label{e:psia}
\Psi(\alpha)\colon \cF_\FF^0(X_{d(\alpha)};\bH_{d(\alpha)})\ra  \cF_\FF^0(X_{c(\alpha)};\bH_{c(\alpha)}),
\end{equation}
defined by,
\begin{equation}\label{e:psias}
\Psi(\alpha)\sum_{i=1}^n \delta_{x_i}h_i:= \sum_{i=1}^n \delta_{\alpha\cdot x_i}h_i,
\end{equation}
where $x_1,\ldots,x_n\in X_{d(\alpha)}$, $h_i\in\cH_{x_i}$, and $ i=1,\ldots,n,\ n\in\NN$ are arbitrary,
with notation as in \eqref{e:delta} and \eqref{e:fesum}.

\begin{remark}\label{r:orbit}
 If $x\in X$, let $\Gamma\cdot x:=\{\alpha\cdot x\mid \alpha\in \Gamma_{a(x)}\}\cup\{x\}$ be the 
$\Gamma$-orbit of $x$.
The assumption \eqref{e:orbit} means that the Hilbert spaces $\cH_x$ are the same throughout the 
$\Gamma$-orbit of $x$. Equivalently, this means that the Hilbert space bundle $\Gamma\!\cdot\! x\ast \bH$ is trivial. 
In order for the definition of $\Psi(\alpha)$ as in \eqref{e:psias} 
to make sense, this assumption is needed.
\end{remark}

\begin{lemma}\label{l:psi} With notation as before and assumption \eqref{e:orbit}, by
\eqref{e:psia} and \eqref{e:psias} we define a semigroupoid morphism 
$(\mathrm{Id}_S,\Psi)$, where $\Psi\colon \Gamma\ra \Gamma_{\bD}$ and $\Gamma_{\bD}$ 
is defined as in Example~\ref{ex:semes}, with $\bD=\{\cD_s\}_{s\in S}$ and
\begin{equation}\label{e:cedes}
\cD_s:=\cF_\FF^0(X_s;\bH_s),\quad s\in S.
\end{equation}
\end{lemma}

\begin{proof} Let $(\alpha,\beta)\in\Gamma^{(2)}$, hence $c(\beta)=d(\alpha)$. For any 
$f\in \cF_\FF^0(X_{d(\beta)};\bH_{d(\beta)})$ there exist uniquely $x_1,\ldots,x_n\in X_{d(\beta)}$ and $h_i\in \cH_{x_i}$, for all 
$i=1,\ldots,n$, such that
\begin{equation*}
f=\sum_{i=1}^n \delta_{x_i}h_i,
\end{equation*}
and then,
\begin{align*} 
\Psi(\alpha\beta)f & = \sum_{i=1}^n \delta_{\alpha\beta\cdot x_i}h_i=  \sum_{i=1}^n \delta_{\alpha\cdot(\beta\cdot x_i)}h_i \\
& =  \Psi(\alpha)\sum_{i=1}^n \delta_{\beta\cdot x_i}h_i = \Psi(\alpha)\Psi(\beta)  \sum_{i=1}^n \delta_{x_i}h_i=\Psi(\alpha)\Psi(\beta)f.\qedhere
\end{align*}
\end{proof}

The semigroupoid representation $(\mathrm{Id}_S,\Psi)$ as in Lemma~\ref{l:psi} is called the 
\emph{$\Gamma$-shift representation} on the bundle $X\ast \bH$ on which $\Gamma$ acts.
The kernel $K$ is called \emph{invariant} under the action of $\Gamma$ on $X$, or simply 
\emph{$\Gamma$-invariant}, if
\begin{equation}\label{e:invariant}
K(\alpha\cdot x,y)=K(x,\alpha^*\cdot y),\quad x,y\in X,\ \alpha\in \Gamma_{a(x)}^{a(y)}.
\end{equation}
Note that, in order for the equality \eqref{e:invariant} to make sense, 
the assumption \eqref{e:orbit} is needed, and sufficient (since the involution $*$ is a 
bijection on $\Gamma$).

Our first goal is to show that any $\Gamma$-invariant $\bH$-operator valued kernel $K$ yields a
certain generalised $*$-representation of $\Gamma$ on a bundle of Hilbert space linearisations associated to $K$.
First, we need a definition, which extends to the present setting the idea of representations with unbounded operators,
cf.\ \cite{Schmudgen}, \cite{GheondeaUdrea}.

\begin{definition}\label{d:invariantlin} 
Let $K$ be a $\Gamma$-invariant kernel with respect to the bundle of Hilbert spaces 
$X\ast \bH$. A quadruple $(\bK;\bD;\bV;\Psi)$ is called a \emph{generalised
$\Gamma$-invariant Hilbert space linearisation of $K$} if the following properties hold.
\begin{itemize}
\item[(GIL1)] $\bK=\{\cK_s\}_{s\in S}$ is a 
bundle of Hilbert spaces, $\bD=\{\cD_s\}_{s\in S}$ is a bundle of vector 
spaces, such that $\cD_s$ is dense in $\cK_s$, for each $s\in S$.
\item[(GIL2)] $\bV=\{V_x\}_{x\in X}$ is a bundle of bounded linear operators
$V_x\colon \cH_x\ra \cK_{a(x)}$, for each $x\in X$, such that $K(x,y)=V_x^*V_y$, for all $x,y\in X$ such 
that $a(x)=a(y)$.
\item[(GIL3)] $(\mathrm{Id}_S;\Psi)$ is a generalised $*$-representation of $\Gamma$ on $\Gamma_{\bK,\bD}$
and $\Psi(\alpha) V_x=V_{\alpha\cdot x}$, for all $x\in X$ and all $\alpha\in\Gamma_{a(x)}$.
\end{itemize}
The generalised $\Gamma$-invariant linearisation  $(\bK;\bD;\bV;\Phi)$ is called \emph{minimal} if
\begin{itemize}
\item[(GIL4)] For each $s\in S$, the Hilbert space $\cK_s$ is the closed span of 
$\{V_x \cH_x\mid x\in X,\ a(x)=s\}$.
\end{itemize}
\end{definition}

\begin{theorem}\label{t:invariantul}
Let $K$ be a $\Gamma$-invariant kernel which is partially positive semidefinite, in the sense 
of \eqref{e:partpsd}, with respect to the partition of $X$ as in \eqref{e:part}. Then, 
there exists a minimal generalised $\Gamma$-invariant Hilbert space 
linearisation of $K$, that we denote by $(\bK;\bD;\bV;\Psi)$, and 
which is unique modulo unitary equivalence in the following sense: whenever 
$(\bK^\prime;\bD^\prime;\bV^\prime;\Psi^\prime)$ is another generalised 
minimal $\Gamma$-invariant linearisation of $K$, 
with $\cD^\prime_s=\lin\{ V^\prime_x\cH_x\mid x\in X_s\}$, for all $s\in S$,
there exists a bundle of unitary operators $\bU=\{U_s\}_{s\in S}$, subject to the following conditions.
\begin{itemize} 
\item[(a)] $U_s\colon \cK_s\ra \cK_s^\prime$ and $U_s\cD_s\subseteq \cD_s^\prime$, for all $s\in S$, 
\item[(b)] $U_{c(\alpha)} \Psi(\alpha)|_{\cD_{d(\alpha)}}=\Psi(\alpha) U_{d(\alpha)}|_{\cD_{d(\alpha)}}$, 
for all $\alpha\in \Gamma$, 
\item[(c)]  $U_{a(x)} V_x=V_x^\prime$, for all $x\in X$.
\end{itemize}
\end{theorem}

\begin{proof} For each $s\in S$, consider the Hilbert space bundle $X_s\ast \bH_s$, 
where $X_s$ is defined as in \eqref{e:xes} and $\bH_s:=\{\cH_x\}_{x\in X_s}$, and note that, by the very definition of partial 
positive definiteness of $K$ with respect to the partition $P_a=\{X_s\}_{s\in S}$, 
this makes the kernel $K^s:=K|_{X_s\times X_s}$
positive semidefinite. Then, we apply the construction of the $\bH_s$-valued Hilbert space linearisation of the kernel $K^s$, as 
described in the proof of the implication (a)$\Ra$(b), let us denote it by $(\cK_s,V_s)$. More precisely, on the vector space
$\cF_\FF^0(X_s;\bH_s)$, we have the convolution operator $C_{K^s}\colon \cF_\FF^0(X_s;\bH_s)\ra \cF_\FF(X_s;\bH_s)$,
defined as in \eqref{e:cekaw}, the positive semidefinite inner product $\langle\cdot,\cdot\rangle_{K^s}$. defined as in 
\eqref{e:ipk}, and then, considering the vector space
\begin{equation*}
\cN_{K^s}:=\{f\in \cF_\FF^0(X_s;\bH_s)\mid \langle f,f\rangle_{K^s}=0\}=\ker(C_{K^s}),
\end{equation*}
there exists a unique positive definite inner product  on 
$\cF_\FF^0(X_s;\bH_s)/\cN_{K^s}$, denoted by the same symbol 
$\langle\cdot,\cdot\rangle_{K^s}$, which is completed to a Hilbert space that we denote by $\cK_s$. We define
\begin{equation}\label{e:cedese}
\cD_s:= \cF_\FF^0(X_s;\bH_s)/\cN_{K^s},\quad s\in S.
\end{equation}
For each $x\in X_s$, the 
bounded operator $V_s(x)\colon \cH_x\ra\cK_s$ is defined by
\begin{equation*}
V_s(x)h:=\delta_x h+\cN_s\in \cF_\FF^0(X_s;\bH_s)/\cN_{K^s}\subseteq \cK_s.
\end{equation*}

We now consider the $\Gamma$-shift representation $(\mathrm{Id}_S,\Psi)$ defined at \eqref{e:psias}. First we prove that,
for every $f\in \cF_\FF^0(X_{d(\alpha)};\bH_{d(\alpha)})$ and $g\in  \cF_\FF^0(X_{c(\alpha)};\bH_{c(\alpha)})$, we have
\begin{equation}\label{e:lapsi}
\langle \Psi(\alpha)f,g\rangle_{K}=\langle f,\Psi(\alpha^*)\rangle_K.
\end{equation}
Indeed, due to linearity, it is sufficient to prove \eqref{e:lapsi} for $f=\delta_x h$ and $g=\delta_y g$, 
for $x\in X_{d(\alpha)}$, $h\in\cH_x$, $y\in X_{c(\alpha)}$, and $k\in \cH_y$. Then, on the one hand, 
\begin{align}
\langle \Psi(\alpha)f,g\rangle_{K} & = \langle \Psi(\alpha)\delta_x h,\delta_y k\rangle_K   
= \langle \delta_{\alpha\cdot x}h,\delta_y k\rangle_K \nonumber \\ 
& = \langle C_K \delta_{\alpha\cdot x}h,\delta_y k\rangle_0  \nonumber 
 = \langle \sum_{z\in X_{c(\alpha)}} K_z\delta_{\alpha \cdot x}(z)h,\delta_y k\rangle_0  \nonumber \\
& = \langle K_{\alpha \cdot x}h, \delta_{y} k\rangle_0 
 = \sum_{t\in X_{c(\alpha)}} \delta_{y}(t) \langle K(t,\alpha\cdot x)h,k\rangle_{\cH_{c(\alpha)}}  \nonumber \\
& = \langle K(y,\alpha\cdot x)h,k\rangle_{\cH_{c(\alpha)}}\nonumber \\
\intertext{which, taking into account that $K$ is $\Gamma$-invariant, equals}
&\label{e:lapsir} = \langle K(\alpha^*\cdot y,x)h,k\rangle_{\cH_{c(\alpha)}}.
\end{align}
On the other hand,
\begin{align}
\langle f,\Psi(\alpha^*)g\rangle_{K} & = \langle \delta_x h,\Psi(\alpha^*)\delta_y k\rangle_K   
= \langle \delta_{x}h,\delta_{\alpha^*\cdot y} k\rangle_K\nonumber  \\ 
& = \langle C_K \delta_{x}h,\delta_{\alpha^*\cdot y} k\rangle_0  \nonumber 
 = \langle \sum_{z\in X_{d(\alpha)}} K_z\delta_{x}(z)h,\delta_{\alpha^*\cdot y} k\rangle_0  \nonumber \\
& = \langle K_{x}h, \delta_{\alpha^*\cdot y} k\rangle_0  \nonumber 
 = \sum_{t\in X_{d(\alpha)}} \delta_{\alpha^*\cdot y}(t) \langle K(t,x)h,k\rangle_{\cH_{c(\alpha)}}  \nonumber \\
& = \langle K(\alpha^*\cdot y,x)h,k\rangle_{\cH_{c(\alpha)}}. \label{e:lapsira}
\end{align}
From \eqref{e:lapsira} and \eqref{e:lapsir}, we get \eqref{e:lapsi}.

Next we prove that
\begin{equation}\label{e:psian}
\Psi(\alpha)\cN_{K^{d(\alpha)}}\subseteq \cN_{K^{c(\alpha)}}.
\end{equation}
Indeed, let $f\in \cF_\FF^0(X_{d(\alpha)};\bH_{d(\alpha)})$ be arbitrary, hence, for some $x_1,\ldots,x_n\in X_{d(\alpha)}$ and 
$h_i\in\cH_{x_i}$, $f$ has the 
representation \eqref{e:fesum}. If $f\in \cN_{K^{d(\alpha)}}$, by definition, this means that, for any $x\in X_{d(\alpha)}$, we have
\begin{align} \nonumber
0 & = (C_{K^{d(\alpha)}}f)(x) = \sum_{y\in X_x} K^{d(\alpha)}(x,y)f(y) \nonumber \\
 & = \sum_{y\in X_{d(\alpha)}} K(x,y)\sum_{i=1}^n \delta_{x_i}(y)h_i 
 = \sum_{y\in X_{d(\alpha)}} \sum_{i=1}^n \delta_{x_i}(y)K(x,y)h_i \nonumber \\
& = \sum_{i=1}^n \sum_{y\in X_{d(\alpha)}} 
\delta_{x_i}(y)K(x,y)h_i=
\sum_{i=1}^n K(x,x_i)h_i.\label{e:zeco}
\end{align}
Then, for every $z\in X_{c(\alpha)}$, we have
\begin{align*}
(C_{K^{c(\alpha)}}\Psi(\alpha)f)(z) & = \bigr(C_{K^{c(\alpha)}}\sum_{i=1}^n \delta_{\alpha\cdot x_i}\bigl)(z) 
= \sum_{y\in X^{c(\alpha)}} K(z,y)\sum_{i=1}^n \delta_{\alpha\cdot x_i} h_i \\
& = \sum_{y\in X^{c(\alpha)}} \sum_{i=1}^n K(z,y)\delta_{\alpha\cdot x_i}h_i = \sum_{i=1}^n K(z,\alpha\cdot x_i)h_i,\\
\intertext{which, since $K$ is $\Gamma$-invariant, in the sense of \eqref{e:invariant}, and taking into account \eqref{e:zeco} and of the fact that $\alpha^*\cdot z\in X_{d(\alpha)}$, it equals}
& =\sum_{i=1}^n K(\alpha^*\cdot z,x_i)h_i=0.
\end{align*}
This proves \eqref{e:psian} and from here, it follows that $\Psi(\alpha)$ factors to a linear operator
\begin{equation}
\Psi(\alpha)\colon \cF_\FF^0(X_{d(\alpha)};\bH_{d(\alpha)})/\cN_{K^{d(\alpha)}}\ra  \cF_\FF^0(X_{c(\alpha)};\bH_{c(\alpha)})/
\cN_{K^{c(\alpha)}},
\end{equation}
such that \eqref{e:lapsi} is lifted as well. 

It remains to prove that $\Psi(\alpha) V_x=V_{\alpha\cdot x}$, for all $x\in X$ and all $\alpha\in\Gamma_{a(x)}$. But this is a 
consequence of the definitions of $\Psi$ and $V_x$, more precisely, for any $x\in X$, $\alpha\in\Gamma_{a(x)}$, and 
$h\in \cH_x$, we have
\begin{equation}\label{e:psivex}
\Psi(\alpha) V_xh=\Psi(\alpha) (\delta_x h+\cN_{K^{d(\alpha)}}) =\delta_{\alpha\cdot x}h+\cN_{K^{c(\alpha)}}=V_{\alpha\cdot x}h.
\end{equation}
We have proven that the quadruple $(\bK;\bD;\bV;\Psi)$ is a generalised
$\Gamma$-invariant Hilbert space linearisation of $K$.

For the uniqueness part, we follow the usual construction of the bundle of unitary operators, only that now this 
is done locally. Let $(\bK^\prime;\bD^\prime;\bV^\prime;\Psi^\prime)$ be 
another generalised minimal $\Gamma$-invariant linearisation of $K$ with 
$\cD^\prime_s=\lin\{ V^\prime_x\cH_x\mid x\in X_s\}$. Then, for any $s\in S$, $n,m\in \NN$, 
$x_1,\ldots,x_n,y_1,\ldots,y_m\in X_s$, $h_j\in \cH_{x_j}$ for all $j=1,\ldots,n$, and $g_i\in \cH_{y_i}$ for all
$i=1,\ldots,m$, we have
\begin{align*}
\langle \sum_{j=1}^n V_{x_j}h_j,\sum_{i=1}^m V_{y_i}g_i\rangle_{\cK_s} & 
=   \sum_{j=1}^n \sum_{i=1}^m \langle V_{x_j}h_j,V_{y_i}g_i\rangle_{\cK_s} 
= \sum_{j=1}^n \sum_{i=1}^m \langle V_{y_i}^*V_{x_j}h_j,g_i\rangle_{\cH_{y_i}} \\
& = \sum_{j=1}^n \sum_{i=1}^m \langle K(y_i,x_j)h_j,g_i\rangle_{\cH_{y_i}} \\
& = \sum_{j=1}^n \sum_{i=1}^m \langle V_{y_i}^{\prime *}V_{x_j}^\prime h_j,g_i\rangle_{\cH_{y_i}}
 = \langle \sum_{j=1}^n V_{x_j}^\prime h_j,\sum_{i=1}^m V_{y_i}^\prime g_i\rangle_{\cK_s}. 
\end{align*}
This shows that the operator 
\begin{equation*}U_s \colon \lin\{V_x h\mid x\in X_s,\ h\in \cH_{x}\}\ra
 \lin\{V_x^\prime h\mid x\in X_s,\ h\in \cH_{x}\}\end{equation*} 
 is correctly defined by
 \begin{equation*}
 U_s (\sum_{j=1}^n V_{x_j}h_j):= \sum_{j=1}^n V_{x_j}^\prime h_j,
 \end{equation*}
 for arbitrary $n\in\NN$, $x_1,\ldots,x_n\in X_s$, and $h_j\in\cH_{x_j}$, for all $j=1,\ldots,n$, and it is an isometry 
 with dense range. Consequently, it has a unique extension to a unitary operator 
 $U_s\colon \cK_s\ra\cK^\prime_s$.
 
 Finally, by this definition, it is clear that $U_{a(x)} V_x=V_x^\prime$ for all $x\in X$. Also, for any 
 $\alpha\in \Gamma$ and any $x\in X_{d(\alpha)}$, we have
 \begin{equation*}
 U_{c(\alpha)} \Psi(\alpha) V_{x}= U_{c(\alpha)} V_{\alpha\cdot x}=V^\prime_{\alpha\cdot x} = \Psi^\prime(\alpha)
 V_{x}.\qedhere
 \end{equation*}
\end{proof}

Our second goal is to show that $\Gamma$-invariant kernels yield certain generalised
$*$-representations of the $*$-semigroupoid $\Gamma$ on certain bundles of reproducing kernel Hilbert spaces, which should 
be a counterpart to Theorem~\ref{t:invariantul}, in view of Remark~\ref{r:kolrk}. However, we will provide an independent proof 
which will shed some light on a duality that exists between Hilbert space linearisations and reproducing kernel 
Hilbert spaces.

\begin{theorem}\label{t:invariantu}
Let $K$ be a $\Gamma$-invariant kernel, in the sense that both \eqref{e:orbit} and \eqref{e:invariant} hold, 
which is partially positive semidefinite with respect to the 
partition of $X$ as in \eqref{e:part}.
Then, there exists a unique generalised $*$-representation $(\mathrm{id}_S,\Phi)$ in the sense of 
Definition~\ref{d:urep}, where 
$\Phi\colon \Gamma\ra \Gamma_{\bR,\bD}$, see Example~\ref{ex:semuh},
$\bR=\{\cR_{K^s}\}_{s\in S}$, with $K^s=K|_{X_s\times X_s}$ and $\cR_{K^s}$ denoting the reproducing kernel
Hilbert space defined by the kernel $K^s$, and $\bD=\{\cD_s\}_{s\in S}$, with $\cD_s=\lin\{K^s_x\mid x\in X_s\}$,
such that, for any $\alpha\in \Gamma$ and any $x\in X_{d(\alpha)}$, we have 
\begin{equation}\label{e:fiva}
K^{c(\alpha)}_{\alpha \cdot x}=\Phi(\alpha) K^{d(\alpha)}_x.
\end{equation}
\end{theorem}

\begin{proof} For a fixed but arbitrary $s\in S$, we consider the set $X_s=a^{-1}(\{s\})$ which is nonempty, since 
the anchor map $a$ is supposed to be surjective. Then, the kernel $K^s=K|_{X_s\times X_s}$ is positive 
semidefinite, as assumed in \eqref{e:partpsd}. By Theorem~\ref{t:rkh}, let $\cR_{K^s}$ be the reproducing kernel 
Hilbert space, with respect to the bundle of Hilbert spaces $\bH_s=\{\cH_x\mid x\in X_s\}$. We now use 
notations and facts from Section~\ref{s:psovk}. Let $\cF_\FF(X_s;\bH_s)$ denote the vector space of all 
sections, that is, $f\colon X_s\ra \bigcup_{x\in X_s} \cH_x$ such that $f_x=f(x)\in\cH_x$ for all $x\in X_s$, 
and let $\cF_\FF^0(X_s,\bH_s)$ be its subspace made up by all sections with finite support. Then, letting 
$C_{K^s}$ denote the convolution operator defined as in \eqref{e:conv} and associated to the kernel $K^s$, we have
\begin{align} \nonumber
\cD_s &:=\lin \{K^s_x\mid x\in X_s\}=\ran(C_{K^s}) \\
& = \{f\in \cF_\FF(X_s,\bH_s)\mid f=C_{K^s} g,\mbox{ for some }g\in \cF_\FF^0(X_s,\bH_s)\}. \label{e:des}
\end{align}
Since $\cD_s$ is dense in $\cR_{K^s}$, it follows that the bundle $\bR=\{\cR_{K^s}\}_{s\in S}$ of Hilbert spaces 
and the bundle $\bD=\{\cD_s\}_{s\in S}$ of subspaces 
satisfy the assumptions as in Example~\ref{ex:semuh} and hence,
we can define the $*$-semigroupoid $\Gamma_{\bR,\bD}$.

Let $\alpha\in \Gamma$. We define the linear operator $\Phi(\alpha)\colon \cD_{d(\alpha)}\ra \cD_{c(\alpha)}$ by
\begin{equation}\label{e:phi}
\Phi(\alpha)f:=(f_{\alpha^*\cdot x})_{x\in X_{c(\alpha)}},\quad f\in \cD_{d(\alpha)}.
\end{equation}
In the following we prove that $(\mathrm{Id}_S,\Phi)$ is a generalised $*$-representation of 
the $*$-semigroupoid $\Gamma$ on $\Gamma_{\bR,\bD}$, in the sense of Definition~\ref{d:urep}, and
we show that $\Phi$ has the property \eqref{e:fiva}.

Indeed, with notation as in \eqref{e:deltax}, an arbitrary 
section $f\in \cD_{c(\alpha)}$ has the representation
\begin{equation*}
f=\sum_{j=1}^n C_{K^{c(\alpha)}} \hat h_j,
\end{equation*}
for some $n\in \NN$ and $h_1,\ldots,h_n$, such that, $h_j\in \cH_{y_j}$, for a certain $y_j\in X_{c(\alpha)}$, 
$j=1,\ldots,n$.
Then, for each $x\in X_{c(\alpha)}$, we have
\begin{align*}
f_{\alpha^*\cdot x} & = \sum_{j=1}^n K_{y_j}(\alpha^*\cdot x)h_j = \sum_{j=1}^n K(\alpha^*\cdot x,y_j) h_j \\
\intertext{and, using property \eqref{e:invariant}, this equals}
& = \sum_{j=1}^n K(x,\alpha\cdot y_j)h_j = \sum_{j=1}^n K_{\alpha\cdot y_j}(x)h_j.
\end{align*}
In view of the definition of the subspaces $\cD_s$ for $s\in S$, 
this shows that the range of $\Phi(\alpha)$ is contained in $\cD_{c(\alpha)}$, hence the axiom (UR2) is satisfied.
Also, in particular, for $n=1$, we get that \eqref{e:fiva} holds. 

We now prove that $\Phi$ is multiplicative in the sense of the axiom (UR3). To this 
end, let $(\alpha,\beta)\in \Gamma^{(2)}$, $f\in \cD_{d(\beta)}$ and $g=\Phi(\beta)f$. The latter means that, for 
any $y\in X_{c(\beta)}$, we have $g_y=f_{\beta^*\cdot y}$ and then, $\Phi(\alpha)g=(g_{\alpha^*\cdot z})_{z\in 
X_{c(\alpha)}}$. From here we get that, for any $z\in X_{c(\alpha)}$, we have
\begin{equation}\label{e:galp}
g_{\alpha^*\cdot z}=f_{\beta^*\alpha^*\cdot z}=f_{(\alpha\beta)^*\cdot z}=(\Phi(\alpha\beta)f)_z,
\end{equation}
which proves that $\Phi(\alpha\beta)|_{\cD_{d(\beta)}}=\Phi(\alpha)\Phi(\beta)|_{\cD_{d(\beta)}}$.

For arbitrary $\alpha\in \Gamma$, 
$f\in \cD_{d(\alpha)}$, and $f^\prime \in \cD_{c(\alpha)}$, we have $f=C_{K^{d(\alpha)}} g$ and $f^\prime =
C_{K^{c(\alpha)}} g^\prime$, for some $g\in \cF_\FF^0(X_{d(\alpha)};\bH_{d(\alpha)})$ and
$g^\prime\in \cF_\FF^0(X_{c(\alpha)};\bH_{c(\alpha)})$. Then,
\begin{align}\nonumber
\langle \Phi(\alpha)f,f^\prime\rangle_{\cR_{K^{c(\alpha)}}} & = 
\sum_{y\in X_{c(\alpha)}} \langle f_{\alpha^*\cdot y},g^\prime_y\rangle_{\cH_y} \\
& =\sum_{y\in X_{c(\alpha)}} \sum_{z\in X_{d(\alpha)}} \langle K(\alpha^*\cdot y,z)g_z,g^\prime_y\rangle_{\cH_y} \nonumber\\
\intertext{which, in view of \eqref{e:invariant}, equals}
& =\sum_{y\in X_{c(\alpha)}} \sum_{z\in X_{d(\alpha)}} \langle K(y,\alpha\cdot z)g_z,g^\prime_y\rangle_{\cH_y}\nonumber \\
\intertext{which, since $K$ is Hermitian, equals}
& =\sum_{y\in X_{c(\alpha)}} \sum_{z\in X_{d(\alpha)}} \langle g_z,K(\alpha\cdot z,y)g^\prime_y\rangle_{\cH_z} \nonumber\\
\intertext{and, since both sums have a finite number of terms, we can change their order}
&  =\sum_{z\in X_{d(\alpha)}} \sum_{y\in X_{c(\alpha)}} \langle g_z,K((\alpha^*)^* g^\prime)g^\prime_y\rangle_{\cH_z}\nonumber 
\\ & = \sum_{z\in X_{d(\alpha)}} \langle g_z,f^\prime_{\alpha\cdot z}\rangle_{\cH_z} = \langle f,\Phi(\alpha^*)f^\prime
\rangle_{\cR_{K^{d(\alpha)}}}.\label{e:lapa}
\end{align}
This shows that the axioms (UR6) and (UR7) hold as well.

The uniqueness part is performed with all details in the previous theorem, which is the counterpart of this theorem in terms of 
minimal linearisations, see Remark~\ref{r:kolrk}, so we do not repeat it.
\end{proof}

\subsection{$\Gamma$-Invariant Hilbert Space Linearisations.}\label{ss:ihsl}
The next question that we address here refers to the possibility that the generalised $*$-representation $\Phi$ in 
Theorem~\ref{t:invariantu} is performed through bounded operators.

\begin{definition}\label{d:invariantlinb}
Let $K$ be a Hermitian and $\Gamma$-invariant kernel with respect to the bundle of Hilbert spaces 
$X\ast \bH$. A triple $(\bK;\bV;\Phi)$ is called a
\emph{$\Gamma$-invariant Hilbert space linearisation of $K$} if the following properties hold.
\begin{itemize}
\item[(IL1)] $\bK=\{\cK_s\}_{s\in S}$ is a bundle of Hilbert spaces.
\item[(IL2)] $\bV=\{V_x\}_{x\in X}$ is a bundle of bounded linear operators
$V_x\colon \cH_x\ra \cK_{a(x)}$, for each $x\in X$, such that $K(x,y)=V_x^*V_y$, for all $x,y\in X$ such 
that $a(x)=a(y)$.
\item[(IL3)] $(\mathrm{Id}_S;\Phi)$ is a $*$-representation of $\Gamma$ on $\Gamma_{\bK}$,
such that $\Phi(\alpha) V_x=V_{\alpha\cdot x}$, for all $x\in X$ and all $\alpha\in\Gamma_{a(x)}$.
\end{itemize}
The $\Gamma$-invariant Hilbert space linearisation  
$(\bK;\bV;\Phi)$ is called \emph{minimal} if
\begin{itemize}
\item[(IL4)] For each $s\in S$ the Hilbert space $\cK_s$ is the closed span of $\{V_x \cH_x\mid x\in X,\ a(x)=s\}$.
\end{itemize}
\end{definition}

When $\Gamma$-invariant representations are expected, a boundedness condition is natural to show up. This condition first 
appeared in the seminal paper of B.~Sz.-Nagy \cite{BSzNagy}, as well as in many of its generalisations, 
\cite{ConstantinescuGheondea2}, \cite{AyGh2015}, \cite{AyGh2017}, \cite{AyGh2019}, \cite{GheondeaUdrea}. 
We first introduce a definition that will be used in different places, from now on. 

\begin{lemma}\label{l:bst} Let $X\ast \bH$ be a bundle of
Hilbert spaces over the same field $\FF$, indexed on the nonempty set $X$, let $\Gamma$ be a $*$-semigroupoid and 
$(a;\cdot)$ a left action of $\Gamma$ on $X$ such that \eqref{e:orbit} holds, that is, for arbitrary $x\in X$, letting 
$\Gamma\cdot x=\{\alpha\cdot x\mid \alpha\in \Gamma_{a(x)}\}\cup\{x\}$ be the orbit of $x$, the 
bundle $\{\cH_y\}_{y\in \Gamma\cdot x}$ is trivial. The following two assertions are equivalent.

\nr{1}  For all 
$\alpha\in\Gamma$, there exists $M_\alpha\geq 0$ such that, for any section $g$ from 
$\cF_\FF^0(X_{d(\alpha)};\bH_{d(\alpha)})$, 
we have
\begin{equation}\label{e:bounded}
\sum_{x,y\in X_{d(\alpha)}} \langle K(\alpha\cdot x,\alpha\cdot y)g_y,g_x\rangle_{\cH_{x}} \leq
M_\alpha \sum_{x,y\in X_{d(\alpha)}} \langle K(x,y)g_y,g_x\rangle_{\cH_{x}}.
\end{equation}

\nr{2} For all $\alpha\in \Gamma$, the $\Gamma$-shift operator $\Psi(\alpha)$, defined as in \eqref{e:psias}, is bounded with 
respect to the seminorms defined by the positive semidefinite inner product $\langle\cdot,\cdot\rangle_K$ on 
$\cF_\FF^0(X_{d(\alpha)};\bH_{d(\alpha)})$ and $\cF_\FF^0(X_{c(\alpha)};\bH_{c(\alpha)})$, respectively.

If, in addition, $K$ is supposed to be $\Gamma$-invariant, in the sense of definition as in \eqref{e:invariant}, 
then they are equivalent with the following assertion, too.

\nr{3} For all $\alpha\in\Gamma$, the operator $\Phi(\alpha)$ defined at \eqref{e:phi} is bounded, with respect to the strong 
topologies of the reproducing kernel Hilbert spaces $\cR_{K^{d(\alpha)}}$ and $\cR_{K^{c(\alpha)}}$.
\end{lemma}

\begin{proof} (1)$\Leftrightarrow$(2). This follows if we prove that, for any $\alpha\in \Gamma$ and any 
$f,g\in\cF_\FF^0(X_{d(\alpha)};\bH_{d(\alpha)})$, we have
\begin{equation}\label{e:lapas}
\langle \Psi(\alpha)f,\Psi(\alpha)g\rangle_K = \sum_{x\in X_{d(\alpha)}}\sum_{y\in X_{d(\alpha)}}
 \langle K(\alpha\cdot y,\alpha\cdot x)f_x,g_y\rangle_{\cH_y}.
\end{equation}
In view of linearity, it is sufficient to prove \eqref{e:lapas} for $f=\delta_x h$ and $g=\delta_y k$, 
with $x,y\in X_{d(\alpha)}$, $h\in \cH_{x}$ and $k\in \cH_y$. Indeed,
\begin{align*}
\langle \Psi(\alpha)f,\Psi(\alpha)g\rangle_K & = \langle \Psi(\alpha)\delta_x h,\Psi(\alpha)\delta_y k\rangle_K 
=\langle \delta_{\alpha\cdot x}h,\delta_{\alpha\cdot y}k\rangle_K \\
& = \langle C_K\delta_{\alpha\cdot x}h,\delta_{\alpha\cdot y}k\rangle_0
= \langle \sum_{z\in X_{d(\alpha)}} K_z \delta_{\alpha\cdot x}(z)h,\delta_{\alpha\cdot y}k\rangle_0 \\
& = \langle K_{\alpha\cdot x}h,\delta_{\alpha\cdot y}k\rangle_0
= \sum_{z\in X_{d(\alpha)}} \langle K_{\alpha\cdot x}(z),\delta_{\alpha\cdot y}(z)k\rangle_{\cH_y} \\
& = \langle K(\alpha\cdot y,\alpha\cdot x)h,k\rangle_{\cH_y}.
\end{align*}

(1)$\Leftrightarrow$(3). Let $\alpha,\beta\in\Gamma$  and let $f=K_xh$, for some $x\in X_{d(\alpha)}$, $h\in\cH_x$, and 
$g=K_y k$, for some $y\in X_{d(\beta)}$ and $k\in \cH_y$, be arbitrary. Then
\begin{align}\nonumber
\langle \Phi(\alpha)f,\Phi(\beta)g\rangle_\cR & = \langle \Phi(\alpha)K_xh,\Phi(\beta)K_yk\rangle_\cR 
= \langle K_x(\alpha^*\cdot)h,K_y(\beta^*\cdot)k\rangle_\cR \\
\intertext{which, taking into account of the $\Gamma$-invariance of $K$, equals}
& = \langle K_{\alpha\cdot x}h,K_{\beta_\cdot y}k\rangle\cR= \langle K(\beta\cdot y,\alpha\cdot x)h,k\rangle_{\cH_y}.\label{e:pafib}
\end{align}
If $f\in \lin\{K_xh\mid x\in X_{d(\alpha)},\ h\in\cH_x\}=\ran(C_{K^{d(\alpha)}})$, then there exists $x_1,\ldots,x_n\in X_{d(\alpha)}$ 
and $h_i\in \cH_{x_i}$, for $i=1,\ldots,n$, such that
\begin{equation*}
f=\sum_{i=1}^n K_{x_i}h_i.
\end{equation*}
From \eqref{e:pafib} we get
\begin{align}\nonumber
\langle \Phi(\alpha)f,\Phi(\alpha)f\rangle_{\cR_{K^{c(\alpha)}}} & = \sum_{i,j=1}^n \langle \Phi(\alpha)K_{x_j}h_j,\Phi(\alpha)K_{x_i}h_i\rangle_{\cR_{K^{c(\alpha)}}} \\
& = \sum_{i,j=1}^n \langle K(\alpha\cdot x_i,\alpha\cdot x_j)h_j,h_i\rangle_{\cH_{x_i}}\nonumber \\
\intertext{which, in view of the assumption (1), is dominated by}
& \leq M_\alpha \sum_{i,j=1}^n \langle K(x_i,x_j)h_i,h_j\rangle_{\cH_{x_i}} \nonumber \\
& = M_\alpha 
\sum_{i,j=1}^n \langle K_{x_i}h_i,K_{x_j}h_j\rangle_{\cR_{K^{d(\alpha)}}} \nonumber \\
& = M_\alpha \bigl\langle \sum_{i=1}^n K_{x_i}h_i,\sum_{j=1}^n K_{x_j}h_j\
\bigr\rangle_{\cR_{K^{d(\alpha)}}} \nonumber \\
&  = M_\alpha \langle f,f\rangle_{\cR_{K^{d(\alpha)}}}.\label{e:lafif}
\end{align}
This means that $\Phi(\alpha)$ is bounded with respect to the norm induced by the inner products of ${\cR_{K^{c(\alpha)}}}$ and
${\cR_{K^{d(\alpha)}}}$, respectively.

For the converse implication we only have to observe that the reasoning is reversible.
\end{proof}

 A kernel $K$
with property (1) from Lemma~\ref{l:bst} is called \emph{of bounded shift type}. Let us observe that, in order to formulate the 
concept of a kernel of bounded shift type, it is sufficient that $\Gamma$ is a semigroupoid only.

\begin{theorem}\label{t:invariantb}
Let $K$ be a $\Gamma$-invariant kernel which is partially positive semidefinite, in the sense that it is partially Hermitian and
\eqref{e:partpsd} holds, with respect to the partition of $X$ as in \eqref{e:part}. The following assertions are equivalent.

\nr{1} $K$ is of bounded shift type.

\nr{2} There exists a unique $*$-representation $(\mathrm{Id}_S,\Phi)$, in the sense of 
Definition~\ref{d:serep}, where 
$\Phi\colon \Gamma\ra \Gamma_{\bR}$, see Example~\ref{ex:segah},
$\bR=\{\cR_{K^s}\}_{s\in S}$, with $K^s=K|_{X_s\times X_s}$ and $\cR_{K^s}$ denoting the reproducing kernel
Hilbert space defined by the kernel $K^s$, such that, for any $\alpha\in \Gamma$ and any $x\in X_{d(\alpha)}$, the property
\eqref{e:fiva} holds, that is,
\begin{equation*}
K^{c(\alpha)}_{\alpha \cdot x}=\Phi(\alpha) K^{d(\alpha)}_x.
\end{equation*}

\nr{3} There exists a minimal $\Gamma$-invariant linearisation of $K$, that we denote by 
$(\bK;\bV;\Psi)$, and 
which is unique modulo unitary equivalence, in the following sense: whenever 
$(\bK^\prime;\bV^\prime;\Psi^\prime)$ is another generalised 
minimal $\Gamma$-invariant linearisation of $K$, 
there exists a bundle of unitary operators $\bU=\{U_s\}_{s\in S}$, subject to the following conditions.
\begin{itemize}
\item[(a)] $U_s\colon \cK_s\ra \cK_s^\prime$, for all $s\in S$.
\item[(b)] $U_{c(\alpha)} \Psi(\alpha)=\Psi^\prime(\alpha) U_{d(\alpha)}$, 
for all $\alpha\in \Gamma$.
\item[(c)]  $U_{a(x)} V_x=V_x^\prime$, for all $x\in X$.
\end{itemize}
\end{theorem}

\begin{proof}  We inspect the proofs of Theorem~\ref{t:invariantul} and 
of Theorem~\ref{t:invariantu} and observe that, the only thing we have to prove is the boundedness of the
operators $\Phi(\alpha)$ and $\Psi(\alpha)$, respectively, for which we use Lemma~\ref{l:bst}.
\end{proof}

As in \cite{GheondeaUdrea}, the boundedness condition \eqref{e:bounded} holds automatically 
if $\Gamma$ is an inverse semigroupoid.

\begin{corollary}\label{c:invsem} Assume that $\Gamma$ is an inverse semigroupoid and let 
$K$ be a $\Gamma$-invariant kernel which is partially positive semidefinite, in the sense 
of \eqref{e:partpsd}, with respect to the partition of $X$ as in \eqref{e:part}. Then assertions (1) through (3) in 
Theorem~\ref{t:invariantb} are true.
\end{corollary}

\begin{proof} By Theorem~\ref{t:invariantu}, 
there exists a unique generalised $*$-representation $(\mathrm{id}_S,\Phi)$ in the sense of 
Definition~\ref{d:urep}, where 
$\Phi\colon \Gamma\ra \Gamma_{\bR,\bD}$, where $\Gamma_{\bR,\bD}$ is defined as in Example~\ref{ex:semuh},
$\bR=\{\cR_{K^s}\}_{s\in S}$, $K^s=K|_{X_s\times X_s}$, and $\cR_{K^s}$ denoting the reproducing kernel
Hilbert space defined by the kernel $K^s$, and $\bD=\{\cD_s\}_{s\in S}$, with $\cD_s=\lin\{K^s_x\mid x\in X_s\}$,
such that for any $\alpha\in \Gamma$ and any $x\in X_{d(\alpha)}$ we have 
$K^{c(\alpha)}_{\alpha \cdot x}=\Phi(\alpha) K^{d(\alpha)}_x$.

For any element $\alpha\in\Gamma$, we consider the operator 
$\Phi(\alpha)\colon \cD_{d(\alpha)}\ra \cD_{c(\alpha)}$ and we want to prove
that it is bounded.
Indeed,  the operator 
$\Phi(\alpha^*)\colon \cG_{c(\alpha)}\ra\cG_{d(\alpha)}$ has the property 
$\Phi(\alpha^*)\subseteq \Phi(\alpha)^*$, when viewing $\Phi(\alpha)$ as a 
densely defined linear operator from the Hilbert space $\cK_{d(\alpha)}$ to the Hilbert space 
$\cK_{c(\alpha)}$. Then, since $\Gamma$ is an 
inverse semigroupoid, hence $\alpha^*\alpha\alpha^*=\alpha^*$, 
it follows that, for any $h\in\cD_{c(\alpha)}$, we have
\begin{equation*}
\Phi(\alpha)^*\Phi(\alpha)\Phi(\alpha)^*h=\Phi(\alpha^*)\Phi(\alpha)
\Phi(\alpha^*)h=\Phi(\alpha^*\alpha\alpha^*)h=\Phi(\alpha^*)h,
\end{equation*}
hence, for any $k\in\cD_{d(\alpha)}$, letting $h=\Phi(\alpha)k$, we have
\begin{equation*}
\Phi(\alpha)^*\Phi(\alpha)\Phi(\alpha)^*\Phi(\alpha)k=\Phi(\alpha^*)\Phi(\alpha)k,
\end{equation*}
that is, the operator $\Phi(\alpha)^*\Phi(\alpha)\colon \cD_{d(\alpha)}\ra\cD_{d(\alpha)}$ is a 
projection. This implies that $\cD_{d(\alpha)}$ has a decomposition
\begin{equation*}
\cD_{d(\alpha)}=\ker(\Phi(\alpha)^*\Phi(\alpha))\dot+ \ran(\Phi(\alpha)^*\Phi(\alpha)).
\end{equation*}
Also, since $\Phi(\alpha)^*\Phi(\alpha)h=\Phi(\alpha^*\alpha)h$ holds for
all $h\in\cD_{d(\alpha)}$, it follows that
\begin{equation*}
\ker(\Phi(\alpha)^*\Phi(\alpha))\perp \ran(\Phi(\alpha)^*\Phi(\alpha)).
\end{equation*}
From here we get that $\ker(\Phi(\alpha)^*\Phi(\alpha))=\ker(\Phi(\alpha))$. Then, letting
$h\in\cD_{d(\alpha)}$ be arbitrary, hence $h=h_0+h_1$, where $h_0\in\ker(\Phi(\alpha))$ and
$h_1\in\ran(\Phi(\alpha)^*\Phi(\alpha))$, we have
\begin{align*}
\langle \Phi(\alpha)h,\Phi(\alpha)h\rangle_{\cK_{\tau(c(\alpha))}} &  = 
\langle \Phi(\alpha)h_1,\Phi(\alpha)h_1\rangle_{\cK_{\tau(d(\alpha))}} \\
& = \langle \Phi(\alpha)^*\Phi(\alpha)h_1, h_1\rangle_{\cK_{\tau(d(\alpha))}} \\
& =\langle h_1,h_1\rangle_{\cH_{\tau(d(\alpha))}}\leq \langle h,h\rangle_{\cK_{\tau(d(\alpha))}},
\end{align*}
where we have taken into account that $\Phi(\alpha)^*\Phi(\alpha)$ is a projection and acts like
identity operator on its range. Thus, the operator $\Phi(\alpha)$ is a partial isometry, hence bounded,
and it can be extended to a linear bounded operator $\cL_{d(\alpha)}\ra \cL_{c(\alpha)}$, which is a partial isometry. \end{proof}

Since any groupoid is an inverse semigroupoid, the previous corollary holds, as a particular case, for groupoids.

\section{Hermitian Kernels}\label{s:hk}

Hermitian kernels which are not semidefinite cannot produce Hilbert spaces but they may produce spaces with indefinite 
inner product. The most tractable class of indefinite inner product spaces is that of Krein spaces. 
But, if one wants to produce a Krein 
space from a Hermitian kernel, an additional obstruction shows up. This obstruction is explained by the fact that the linear span 
of positive semidefinite kernels is, in general, smaller than the real vector space of Hermitian kernels. Most of these things are 
already investigated in \cite{ConstantinescuGheondea1} and \cite{ConstantinescuGheondea2}. Here, we are interested in those 
Hermitian kernels which 
produce Krein spaces and which are invariant under actions of $*$-semigroupoids. In the following, we briefly 
review what is known about this problem and then, we prove the new results that hold for the actions of $*$-semigroupoids.

\subsection{Krein Spaces and Their Linear Operators} \label{ss:kstlo}
We briefly review the definitions and fundamental facts on Krein spaces and their linear operators, following \cite{GheondeaE}. 
Given a vector space $\cX$ over the field $\FF$, an \emph{inner product} on $\cX$ is a map 
$[\cdot,\cdot]\colon \cX\times \cX\ra \FF$ which is linear in the first variable and (conjugate, in case $\FF=\CC$) symmetric. 
Then it is conjugate linear in the second variable and the inner square $[x,x]$ takes only real values, but they may be 
negative, null, or positive. A vector $x\in\cX$ such that $[x,x]<0$ is called \emph{negative}, if $[x,x]=0$ it is called \emph{neutral}, 
and if $[x,x]>0$ it is called \emph{positive}. Given a linear manifold $\cL\subseteq \cX$, it is called \emph{negative}, 
\emph{neutral}, \emph{positive}, respectively, if all nonnull vectors in $\cL$ are of that kind. The linear manifold $\cL$ 
is called \emph{nonnegative} or \emph{nonpositive}, respectively, if $[x,x]\geq 0$, or $[x,x]\leq 0$, for all $x\in \cL$. Two vectors $x,y\in \cX$ are called \emph{orthogonal}, denoted by $x\perp y$, if $[x,y]=0$.
We denote by $\cL^\perp:=\{y\in \cX\mid [x,y]=0\mbox{ for all }x\in\cL\}$ the \emph{orthogonal companion} of $\cL$ and by 
$\cL^0:=\cL\cap \cL^\perp$ its \emph{isotropic} part. If $\cL^0$ is the null space then $\cL$ is called \emph{nondegenerate} and, 
it is called \emph{degenerate} in the opposite case. The inner product space $(\cX,[\cdot,\cdot])$ is called \emph{indefinite} if it 
contains positive vectors and negative vectors. In this case, nonnull neutral vectors exist as well.

\begin{theorem}\label{t:desc}
Let $(\cH,[\cdot,\cdot])$ be an inner product space. Then the 
following statements are equivalent.

{\em (a)} There exist two linear manifolds $\cH^+$ and $\cH^-$ such that
$\cH^+\perp\cH^-$, $\cH=\cH^++\cH^-$, and both $(\cH^+,[\cdot,\cdot])$ and
$(\cH^-,-[\cdot,\cdot])$ are Hilbert spaces.

{\em (b)} There exists a linear operator $J$ on $\cH$ such that $J^2=I$
and the equality
\begin{equation}\label{e:jinp}\langle x,y\rangle_J=[Jx,y],\quad x,y\in\cH,\end{equation}
defines a positive definite inner product in $\cH$ such that
$(\cH,\langle\cdot,\cdot\rangle_J)$ is a Hilbert space.

{\em (c)} There exists a positive definite inner product
$\langle\cdot,\cdot\rangle$ on $\cH$ 
such that $(\cH,\langle\cdot,\cdot\rangle)$ is a Hilbert space and the
associated norm $\|\cdot\|$ satisfies
\begin{equation}\label{e:norp}\| x\|=\sup_{\|y\|\leq1}|[x,y]|,\quad x\in\cH.\end{equation}
\end{theorem}

An inner product space $(\cH,[\cdot,\cdot])$, satisfying one (hence, all) of the conditions (a), (b), or (c) from 
Theorem~\ref{t:desc}, is called a \emph{Krein space}. In this case, a decomposition as in item (a) is called a \emph{fundamental 
decomposition}, a linear operator $J$ as in item (b) is called a \emph{fundamental symmetry}, while a Hilbert space norm 
$\|\cdot\|$ as in item (c), is called a \emph{unitary norm}. These objects are in one-to-one correspondence with each other. In 
addition, all unitary norms produce the same topology on $\cH$, and we call it the \emph{strong topology}.

Note that, with respect to the Hilbert space structure $(\cH,\langle\cdot,\cdot\rangle_J)$, a fundamental symmetry $J$
is actually a 
symmetry, that is, a selfadjoint and unitary operator, $J^*=J^{-1}=J$, on the Hilbert space $\cH$. Also, 
\begin{equation*}
[Jx,y]=\langle x,y\rangle_J,\quad x,y\in \cH.
\end{equation*}

Let $(\cK_1,[\cdot,\cdot])$ and $(\cK_2,[\cdot,\cdot])$ be Krein
spaces. We consider a linear operator $T$ with domain
$\dom(T)$ dense in $\cK_1$ and range $\ran(T)$ in $\cK_2$. The 
{\em (Krein space) adjoint operator} $T^\sharp $ of $T$ is defined as follows:
 $\dom(T^\sharp )$ is the set of all 
$y\in\cK_2$
for which there exists $z\in\cK_1$, such that $[Tx,y]=[x,z]$, 
for all $x\in\dom(T)$, and for such a $y\in\cK_2$ we let
$T^\sharp y=z$. Equivalently,
\begin{equation}\label{e:adjop}
\dom(T^\sharp)=\{y\in\cK_2\mid [Tx,y]=[x,z]\mbox{ for some }
z\in\cK_1\mbox{ and all }x\in\dom(T)\},
\end{equation}
and then, since $\dom(T)$ is dense, the element $z\in\cK_1$ as in \eqref{e:adjop} is uniquely
determined by $y$ and we let $T^\sharp y=z$.

Let the fundamental symmetries $J_1$ and $J_2$ on $\cK_1$ and $\cK_2$, respectively, be fixed and, consequently, the Hilbert 
spaces $(\cK_1;\langle \cdot,\cdot\rangle_{J_1})$ and $(\cK_2;\langle\cdot,\cdot\rangle_{J_2})$ be fixed as well. Also, let 
the Hilbert space adjoint $T^*$ be defined as usual. The connection between the two adjoint operators is given by
\begin{equation}\label{e:tesh}
\dom(T^\sharp)=J_2\dom(T^*),\quad T^\sharp =J_1 T^*J_2.
\end{equation}

Since strong topologies on Krein spaces are unique, it makes perfectly sense to consider the vector space $\cB(\cK_1,\cK_2)$ of
all linear and bounded operators $T\colon \cK_1\ra \cK_2$. 
Then $\cB(\cK_1,\cK_2)\ni T\mapsto T^\sharp\in\cB(\cK_2,\cK_1)$ has 
all the required properties of an \emph{involution}. In particular, the concepts of \emph{normal}, \emph{selfadjoint}, 
\emph{orthogonal projection}, \emph{isometric}, and \emph{unitary operator}, can be defined with obvious definitions.

An operator $T\in\cB(\cK)$, for some Krein space $\cK$, is called \emph{fundamentally reducible} if there exists a fundamental 
symmetry $J$ on $\cK$ such that $JT=TJ$, equivalently, that $T^\sharp=T^*$, where $T^*$ denotes the Hilbert space adjoint of 
$T$ with respect to $(\cK,\langle\cdot,\cdot\rangle_J)$. Equivalently, if $\cK=\cK^++\cK^-$ is the fundamental decomposition of $\cK$ 
induced by $J$, this means that, the representation of $T$, as a $2\times 2$ operator block matrix with respect to this 
fundamental decomposition, is diagonal. This property is important since it reduces the study of $T$, for example from the point 
of view of spectral properties, to the direct orthogonal sum of the two operators from the diagonal. In case of special operators, 
like normal, selfadjoint, or unitary operators, fundamental reducibility reduces the study of these operators to the study of direct 
orthogonal sums of operators on Hilbert spaces and of the same kind, that is, normal, selfadjoint, or unitary operators, 
respectively.

\subsection{Krein Spaces Induced by Selfadjoint Operators}\label{ss:ksiso}
In this section we need a construction and some of its properties for Krein spaces induced by bounded selfadjoint operators. This 
is a generalisation of a similar construction for positive operators, in which case it is called a \emph{renorming} of a Hilbert space. 
We follow Section 6.1 in \cite{GheondeaE}.
Let $(\cH,\langle\cdot,\cdot\rangle)$ be a Hilbert space and $A=A^*\in\cB(\cH)$. A \emph{Krein space induced} by $A$ is, 
by definition, a pair $(\cK,\Pi)$ subject to the following conditions.
\begin{itemize}
\item[(IKS1)] $(\cK,[\cdot,\cdot])$ is a Krein space.
\item[(IKS2)] $\Pi\colon \cH\ra\cK$ is a linear bounded operator with dense range.
\item[(IKS3)] $\langle Ah,k\rangle=[\Pi h,\Pi k]$, for all $h,k\in \cH$.
\end{itemize}

Two Krein spaces $(\cK_1,\Pi_1)$ and $(\cK_2,\Pi_2)$, induced by the same operator $A$, 
are called \emph{unitarily equivalent} if
there exists a Krein space unitary operator $U\colon\cK_1\ra\cK_2$, such that $U\Pi_1=\Pi_2$.

\begin{example}\label{ex:kas} Let $A=A^*\in\cB(\cH)$ for some Hilbert space $\cH$. Then $A=A_+-A_-$, where 
$A_\pm\in\cB^+(\cH)$ and are 
uniquely determined by the condition $A_+A_-=0$. More precisely, if $E$ is the spectral measure of $A$ 
and we let $E_+=E(0,+\infty)$ and $E_-=E(-\infty,0)$, then $A_+=AE+$ and $A_-=AE_-$. Then we consider the Hilbert spaces 
$\cK_\pm$ induced by $A_\pm$, more precisely, $\cK_\pm$ is the completion of $\ran(E_\pm)$ to a Hilbert space
with respect to positive definite inner product $\langle A_\pm\cdot,\cdot\rangle_\cH$, hence with respect to the norm 
$\|A^{1/2}_\pm\cdot\|$, and let $\cK_A:=\cK_-\oplus \cK_+$. In particular, $\cK_A$ is the Hilbert space completion of 
$\cH\ominus \ker(A)$ with respect to the norm $\||A|^{1/2}\cdot\|_\cH$. 
Letting $J$ be the symmetry on $\cK_A$ defined by this orthogonal 
decomposition of $\cK_A$, we view $\cK_A$ as a Krein space with fundamental symmetry $J$.  More precisely, the polar 
decomposition $A=S_A |A|$ yields a selfadjoint partial isometry $S_A$ and $J$ is the unique lifting of $S_A$, compressed at 
$\cH\ominus\ker(A)=E_+\cH\oplus E_+\cH$, to $\cK$. In particular, the inner product $[\cdot,\cdot]_{\cK_A}$ on 
$\cH\ominus\ker(A)$ acts by
\begin{equation*}
[h,k]_{\cK_A}=\langle A_+h_+,k_+\rangle_{\cH}-\langle A_-h_-,k_-\rangle_{\cH}=\langle Ah,k\rangle_\cH,\quad h,k\in\cH\ominus\ker(A),
\end{equation*}
where, for an arbitrary vector $h\in\cH\ominus\ker(A)$, we have the unique decomposition $h=h_++h_-$, with 
$h_\pm\in E_\pm\cH$.

The operator $\Pi_A\colon \cH\ra\cK$ is 
defined as follows. For each $h\in\cH$, we have the unique decomposition $h=h_-+h_0+h_+$, where $h_\pm\in\ran(E_\pm)$ 
and $h_0\in\ker(A)$. Then $\Pi_A h:=h_-+h_+\in \cH\ominus\ker(A)\subseteq \cK$ 
and it is easy to see that $\Pi_A$ is bounded and has dense range. 
Also, for each $h,k\in \cH$, with notation as before, we have
\begin{equation*}
[\Pi_A h,\Pi_A k]= [h_-+h_+,k_-+k_+]=\langle A_+h_+,k_+\rangle_\cH-\langle A_-h_-,k_-\rangle_\cH=\langle Ah,k\rangle_\cH. 
\end{equation*}
All these show that $(\cK_A,\Pi_A)$ is a Krein space induced by $A$.
\end{example}

\begin{example}\label{ex:kasu} With notation as in Example~\ref{ex:kas}, define $\cB_A:=\ran(|A|^{1/2})$ endowed with the inner 
product $\langle\cdot,\cdot\rangle_{\cB_A}$ defined as follows. For each $h,k\in\cH\ominus \ker(A)$ we let
\begin{equation*}
\langle |A|^{1/2}h,|A|^{1/2}\rangle_{\cB_A}:=\langle h,k\rangle_\cH.
\end{equation*} 
It is easy to see that, with respect to this inner product, $\cB_A$ is a Hilbert space in such a way that 
$|A|^{1/2}\colon \cH\ominus \ker(A)\ra \cB_A$ is a unitary operator.
Then we observe that, 
in the polar decomposition $A=S_A |A|$, the partial isometry $S_A$ yields a symmetry on $\cB_A$. More precisely, we have
\begin{equation*}
S_Ah= h_+-h_-,\quad h=h_-+h_0+h_+\in\cH,\ h_\pm\in E_\pm\cH,\ h_0\in\ker(A).
\end{equation*}
Since $S_A$ commutes with $A$, it commutes with both $|A|$ and $|A|^{1/2}$, hence it leaves 
$\cB_A=\ran(|A|^{1/2})$ 
invariant and $S_A$ is a symmetry with respect to the Hilbert space $(\cB_A,\|\cdot\|_{\cB_A})$. Then, letting
\begin{equation*}
[|A|^{1/2}h,|A|^{1/2}k]_{\cB_A}:= \langle S_Ah,k\rangle_\cH,\quad h,k\in\cH\ominus \ker(A),
\end{equation*}
it follows that $(\cB_A,[\cdot,\cdot]_{\cB_A})$ is a Krein space such that the compression of $S_A$ to $\cB_A$ is 
a fundamental symmetry.

Let $\Pi_{\cB_A}\colon \cH\ra \cB_A$ be defined by
\begin{equation*}
\Pi_{\cB_A}h:=|A|h,\quad h\in\cH,
\end{equation*}
and note that, 
since $|A|=|A|^{1/2}|A|^{1/2}$, we have $\ran(|A|)\subseteq \ran(|A|^{1/2})$, hence $\Pi_{\cB_A}$ is correctly defined. In addition, 
for each $h\in\cH$, we have
\begin{equation*}
\|\Pi_{\cB_A}h\|_{\cB_A}=\||A|h\|_{\cB_A}=\||A|\|^{1/2}|A|^{1/2}h\|_{\cB_A}=\||A|^{1/2}h\|_\cH\leq \||A|^{1/2} \|h\|_\cH,
\end{equation*}
hence $\Pi_{\cB_A}$ is a bounded linear operator. 
Since $\ran(|A|)$ is dense in $\ran(|A|^{1/2})$ it follows that the range of $\Pi_{\cB_A}$ is dense in $\cB_A$. Also, for any 
$h,k\in\cH$ we have
\begin{align*}
[\Pi_{\cB_A}h,\Pi_{\cB_A}k]_{\cB_A} & = [ |A|h,|A|k]_{\cB_A}=\langle S_A |A|^{1/2}|A|^{1/2}h,|A|^{1/2}|A|^{1/2}k\rangle_{\cB_A}\\
& = \langle  |A|^{1/2}S_A|A|^{1/2}h,|A|^{1/2}|A|^{1/2}k\rangle_{\cB_A} \\
& = \langle S_A |A|^{1/2}h,|A|^{1/2}k\rangle_{\cH}=\langle Ah,k\rangle_{\cH}.
\end{align*}
All these prove that $(\cB_A,\Pi_{\cB_A})$ is a Krein space induced by $A$.

But this Krein space induced by $A$ has two more properties. First, it is continuously included in the Hilbert space $\cH$, which 
follows from the fact that the strong topology on $\cB_A$ was defined in such a way that the operator $|A|^{1/2}$, 
when restricted to $\cH\ominus \ker(A$), is a unitary operator of Hilbert spaces. Second, the two Krein spaces $(\cK_A,\Pi_A)$ 
and $(\cB_A,\Pi_{\cB_A})$ are unitarily equivalent. To see this, consider the operator 
$U\colon \cH\ominus \ker(A)\ra \ran(|A|^{1/2})$ defined by $Uh:=|A|h$, $h\in\cH\ominus\ker(A)$. Then, 
for arbitrary $h,k\in\cH\ominus \ker(A)$, we have
\begin{equation*}
[Uh,Uk]_{\cB_A}=[|A|h,|A|k]_{\cB_A}=\langle S_A |A|^{1/2}h,|A|^{1/2}k\rangle_\cH=\langle Ah,k\rangle_\cH=[h,k]_{\cK_A},
\end{equation*}
which shows that $U$ is isometric with respect to the indefinite inner products of $\cB_A$ and $\cK_A$. The 
operator $U$ is also bounded because
\begin{equation*}
\|Uh\|_{\cB_A} = \||A|h\|_{\cB_A}=\||A|^{1/2}h\|_\cH,\quad h\in \cH\ominus\ker(A),
\end{equation*}
in particular, $U$ is isometric with respect to the Hilbert space norms of $\cK_A$ and $\cB_A$, hence bounded. Since, clearly, 
$U$ is densely defined and with dense range, it has a unique extension to a Krein space unitary operator 
$U\colon \cK_A\ra\cB_A$. Also, 
\begin{equation*}
U\Pi_{\cK_A}h=|A| \Pi_{\cK_A}h=|A|(h_++h_-)=|A|h=\Pi_{\cB_A}h,\quad h\in \cH.
\end{equation*}
In conclusion, $U$ establishes a unitary equivalence of the Krein spaces $(\cK_A,\Pi_{A})$ and $(\cB_A,\Pi_{\cB_A})$ induced 
by $A$.
\end{example}

The question of uniqueness modulo a unitary equivalence of Krein spaces induced by selfadjoint operators has the following 
answer, see Theorem~6.1.9 in \cite{GheondeaE}. This result is actually equivalent with a uniqueness result on Krein spaces 
continuously included in a given Hilbert or Krein space, cf.~T.~Hara \cite{Hara}, and variants of it were proven 
in \cite{ConstantinescuGheondea0} and by M.A.~Dritschel in
\cite{Dritschel}. A stronger result, saying that, in case 
nonuniqueness holds, one can find infinitely many nonunitarily equivalent induced Krein 
spaces by the same operator, was proven by B.~\'Curgus and H.~Langer \cite{CurgusLanger}.

\begin{theorem} Let $A\in\cB(\cH)$ be a selfadjoint operator in a Hilbert space $\cH$. The following assertions are equivalent.
\begin{itemize}
\item[(i)] $A$ has an induced Krein space which is unique, modulo unitary equivalence.
\item[(ii)] There exists $\epsilon>0$ such that either $(-\epsilon,0)$ or $(0,\epsilon)$ is in the resolvent set $\rho(A)$ of $A$.
\end{itemize}
\end{theorem}

In this article we need a result that provides a sufficient condition on bounded operators that can be lifted to bounded operators 
between certain Krein spaces induced by selfadjoint operators. This is a generalisation of a classical result of J.~Dieudonn\'e 
\cite{Dieudonne},
P.~Lax \cite{Lax}, W.T.~Reid~\cite{Reid}, and M.G.~Krein \cite{Krein}, obtained by A.~Dijksma, H.~Langer, and H.S.V.~de Snoo
\cite{DLS}, see also Theorem~6.1.10 in \cite{GheondeaE}.

\begin{theorem}\label{t:abetul} 
Assume that $A\in\cB(\cH)$ and $B\in\cB(\cG)$ are bounded selfadjoint operators and let $T\in\cB(\cH,\cG)$
and $S\in\cB(\cG,\cH)$ be such that $BT=S^*A$, equivalently
\begin{equation}\label{e:taheg}
[Th,g]_B=[h,Sg]_A,\quad h\in\cH,\ g\in\cG.
\end{equation} Then, there exist unique operators $\widetilde T\in\cB(\cK_A,\cK_B)$ 
and $\widetilde S\in\cB(\cK_B,\cK_A)$, such that
\begin{equation*}
\widetilde T\Pi_A=\Pi_BT,\quad \widetilde S\Pi_B=\Pi_A S.
\end{equation*} 
\end{theorem}

Since the Krein spaces $(\cK_A,\Pi_A)$ and $(\cB_A,\Pi_{\cB_A})$ are unitarily equivalent, the previous theorem is equivalent 
with the following one.

\begin{theorem}\label{t:abetu} 
Assume that $A\in\cB(\cH)$ and $B\in\cB(\cG)$ are bounded selfadjoint operators and let $T\in\cB(\cH,\cG)$
and $S\in\cB(\cG,\cH)$ be such that $BT=S^*A$, equivalently, \eqref{e:taheg} holds. 
Then, there exist unique operators $\widetilde T\in\cB(\cB_A,\cB_B)$ 
and $\widetilde S\in\cB(\cB_B,\cB_A)$, such that
\begin{equation*}
\widetilde T\Pi_{\cB_A}=\Pi_{\cB_B}T,\quad \widetilde S\Pi_{\cB_B}=\Pi_{\cB_A} S.
\end{equation*} 
\end{theorem}

\subsection{Krein Space Linearisations of Hermitian Kernels}\label{ss:lhk}
In this subsection we use the notation as in Subsection~\ref{ss:nbd}. Recall that by $\fK^h(X;\bH)$ we denote the real vector 
space of all Hermitian $\bH$-operator valued kernels and, by $\fK^+(X;\bH)$, 
the strict convex cone of all positive semidefinite $\bH$-operator valued kernels. If $K\in\fK^h(X;\bH)$ then we use the 
notation $[\cdot,\cdot]_K$ for the Hermitian sesquilinear form, actually, an indefinite inner product, with terminology of the 
previous sections, induced by $K$ on the vector space $\cF_\FF^0(X;\bH)$, of all $\bH$-valued sections of finite support
\begin{equation}\label{e:fegek}
[f,g]_K:=\langle C_K f,g\rangle_0,\quad f,g\in\cF_\FF^0(X;\bH).
\end{equation}
When $L\in\fK^+(X;\bH)$, then we use the notation $\langle\cdot,\cdot\rangle_L$ as in \eqref{e:ipk}, in order to emphasise
the fact that this is a positive semidefinite inner product.

\begin{definition}\label{d:ksl}
 Given $K\in\fK^h(X;\bH)$, a \emph{Krein space
linearisation}, or a 
\emph{Krein space Kolmogorov decomposition}, of $K$ is, by definition, a pair $(\cK;V)$ subject to
the following conditions.
\begin{itemize}
\item[(kkd1)] $\cK$ is a Krein space over $\FF$.
\item[(kkd2)] $V=\{V(x)\}_{x\in X}$ is an 
operator bundle such that $V(x)\in \cB(\cH_x,\cK)$ for all $x\in X$.
\item[(kkd3)] $K(x,y)=V(x)^\sharp V(y)$ for all $x,y\in X$, equivalently, 
\begin{equation*}
\langle K(x,y)h,k\rangle_{\cH_x}=[V(y)h,V(x)k]_\cK,\quad x,y\in X,\ h\in\cH_y,\ k\in\cH_x.
\end{equation*}
\end{itemize}
A linearisation $(\cK;V)$ is called \emph{minimal} if
\begin{itemize}
\item[(kkd4)] $\cK$ is the closed span of $\{V(x)\cH_x\mid x\in X\}$.
\end{itemize}
\end{definition}

Given two kernels $K,L\in\fK^h(X;\bH)$, we say that $L$ \emph{dominates} $K$, and write $K\leq L$, if $L-K\in \fK^+(X;\bH)$.  Equivalently, this means that
\begin{equation}\label{e:fefek}
[f,f]_K\leq [f,f]_L,\quad f\in \cF_\FF^0(X;\bH).
\end{equation}
In particular, if $K\in\fK^h(X;\bH)$ belongs to $\fK^+(X;\bH)$, we equivalently write $K\geq 0$.
We call $K$ and $L$ \emph{disjoint} if the only $P\in\fK^+(X;\bH)$ such that $P\leq L$ and $P\leq K$ is the null kernel.
The next result is Theorem~3.1 in \cite{ConstantinescuGheondea1}, with an addition from Theorem~1.5.3 in 
\cite{GheondeaE}. The equivalence of (1), (1)$^\prime$, (2), (3), and (3)$^\prime$ originates in \cite{Schwartz} under a different, 
but closely related, setting. For later use, we outline the main steps in the 
implications (2)$\Ra$(4) and (4)$\Ra$(3)$^\prime$.

\begin{theorem}\label{t:ksl}
Given $K\in\fK^h(X;\bH)$, the following assertions are equivalent.
\begin{itemize}
\item[(1)] There exists $L\in\fK^+(X;\bH)$, such that $-L\leq K\leq L$.
\item[(1)$^\prime$] There exists $L\in\fK^+(X;\bH)$, such that $K\leq L$.
\item[(2)] There exists $L\in\fK^+(X;\bH)$, such that
\begin{equation}\label{e:fegalu}
|[f,g]_K|\leq [f,f]_L^{1/2} [g,g]_L^{1/2},\quad f,g\in \cF_\FF^0(X;\bH).
\end{equation}
\item[(3)] $K=K_1-K_2$ for some kernels $K_1,K_2\in \fK^+(X;\bH)$.
\item[(3)$^\prime$] $K=K_1-K_2$ for some disjoint kernels $K_1,K_2\in \fK^+(X;\bH)$.
\item[(4)] There exists a Krein space linearisation $(\cK;V)$ of $K$.
\end{itemize}
In addition, if any of these conditions holds, hence all of them hold, then a minimal Krein space linearisation $(\cK;V)$ of $K$ 
exists. 
\end{theorem}

\begin{proof} (2)$\Ra$(4). So, let $L\in\fK^+(X;\bH)$ as in item (2) and, as 
in Subsection~\ref{ss:hsl}, consider $(\cK_L,V_L)$ a minimal linearisation of $L$. By \eqref{e:fegalu} we have
\begin{align}\label{e:nela}
\cN_L:=\ker(C_L) & =\{f\in\cF_\FF^0(X;\bH)\mid \langle f,f\rangle_L=0\}\subseteq 
\cN_K:=\ker(C_K)\\
& =\{f\in \cF_\FF^0(X;\bH)\mid [f,g]_K=0\mbox{ for all }g\in \cF_\FF^0(X;\bH)\}\nonumber
\end{align}
and hence, the inner product $[\cdot,\cdot]_K$ can be 
uniquely factored to $\cF_\FF^0(X;\bH)/\cN_L$ and then extended to an inner product, denoted again by $[\cdot,\cdot]_K$, 
to the Hilbert space $\cK_L$ such that \eqref{e:fegalu} holds. 
By Riesz' Lemma, there exists a unique selfadjoint contraction $G_K\in\cB(\cK_L)$, called the \emph{Gram operator}
of the inner product $[\cdot,\cdot]_K$ with respect to the inner product $\langle\cdot,\cdot\rangle_L$, such that
\begin{equation}\label{e:fegak}
[f,g]_K=\langle G_Kf,g\rangle_L,\quad f,g\in \cK_L.
\end{equation}
We now use the construction as in Example~\ref{ex:kas} in order to produce the Krein space $(\cK_{G_K};\Pi_{G_K})$, and
define $\cK:=\cK_{G_K}$.

The operators $V(x)\colon \cH_x\ra\cK$ are defined similarly as in Subsection~\ref{ss:hsl}, that is, 
for each $x\in X$, let $V(x)\colon \cH_x\ra\cK_{G_K}$ be the operator defined by
\begin{equation}\label{e:vexek}
V(x)h=\Pi_{K}(\widehat h)= \widehat h+\cN_K\in\cK_{G_K},\end{equation}
with notation as in \eqref{e:deltax}. One has to prove that these operators are bounded. To see this, fix $J$ a fundamental 
symmetry on $\cK$, and let $x\in X$ and $h\in\cH_x$ be arbitrary. 
Then, with respect to the unitary norm of $\cK$ associated to the fundamental symmetry $J$ defined before, we have
\begin{align*}
\|V(x)h\|_\cK^2 & = \|\widehat h+\cN_K\|_\cK^2 = \||G_{K}|^{1/2}(\widehat h+\cN_L)\|_{\cK_L}^2\leq \|(\widehat h+\cN_L)\|_{\cK_L}^2 \\
& =  \langle (\widehat h+\cN_L),(\widehat h+\cN_L)\rangle_{\cK_L} = \langle L(x,x)h,h\rangle_{\cH_x}\\
& \leq \|L(x,x)\| \|h\|_{\cH_x}^2,
\end{align*}
which proves that $V(x)$ is a bounded linear operator.
Finally, for arbitrary $x,y\in X$, $h\in\cH_x$, and $k\in\cH_y$, we have
\begin{align*}
\langle V_x^* J V_yk,h\rangle_{\cH_x} & = \langle JV_yk,V_xh\rangle_{J} = [V(y)k,V(x)h]_{\cK} \\
& = [k+\cN_K,h+\cN_K]_\cK = [K(x,y)k,h]_{\cH_x}.
\end{align*}
In this way, we have a minimal Krein space linearisation $(\cK;V)$ of $K$.

(4)$\Ra$(3)$^\prime$. Let $(\cK;V)$ be a Krein space linearisation of $K$. 
Let $[\cdot,\cdot]_\cK$ denote the indefinite inner product of $\cK$, let 
$J$ be a fundamental symmetry on $\cK$,  and consider the positive definite inner product 
$\langle \cdot,\cdot\rangle_J$, that is,
\begin{equation*}
\langle h,k\rangle_J=[J h,k]_\cK,\quad h,k\in \cK.
\end{equation*}
With respect to the Hilbert space $(\cK;\langle\cdot,\cdot\rangle_J)$, consider the Jordan decomposition $J=J_+-J_-$, 
that is, $J_\pm$ is the spectral projection of $\cK$ onto $\ker(I\mp J)$, hence $I=J_++J_-$ and $J_+J_-=0$. 
Define the $\bH$-operator valued kernels
\begin{equation}
K_\pm(x,y):=V(x)^*J_\pm V(y),\quad x,y\in X,
\end{equation}
and note that both of them are positive semidefinite, $K=K_+-K_-$ and let $L:=K_++K_-\in \fK^+(X;\bH)$. We only need to show 
that $K_+$ and $K_-$ are disjoint, which follows from Proposition~16 in \cite{Schwartz}.
\end{proof}

The problem of uniqueness of minimal Krein space linearisations of Hermitian $\bH$-operator valued kernels is more 
complicated than that of positive semidefinite $\bH$-operator valued kernels, as in Theorem~\ref{t:kolmogorov}.
Because of the lack of uniqueness, we will denote
by $(\cK_K;V_K)$ the minimal Krein space linearisation of $K$ constructed during the proof of the implication (2)$\Ra$(4) of
Theorem~\ref{t:ksl}.
First, the definition. Given an $\bH$-operator valued kernel $K$ on $X$, two Krein space linearisations $(\cK_1;V_1)$ and 
$(\cK_2;V_2)$ of $K$ are called \emph{unitarily equivalent}, if there exists a Krein space unitary operator 
$U\colon \cK_1\ra \cK_2$, such that $UV_1(x)=V_2(x)$, for all $x\in X$. 

Also, given $K\in\fK^h(X;\bH)$ and $L\in\fK^+(X;\bH)$ be such that $-L\leq K\leq L$,
letting $\cK_L$ denote the $\bH$-operator valued minimal linearisation of $L$, which is the quotient-completion of
the positive semidefinite inner product space $(\cF_\FF^0(X;\bH);\langle \cdot,\cdot\rangle_L)$ to a Hilbert space $\cK_L$, let
$G_K\in\cB(\cK_L)$ denote the Gram operator of the inner product $[\cdot,\cdot]_K$, that is, the unique selfadjoint and 
contractive operator $G_K\colon \cK_L\ra\cK_L$ such that
\begin{equation}\label{e:gram}
[h,g]_K=\langle G_K h,g\rangle_{\cK_L},\quad h,g\in \cK_L.
\end{equation}
As usually, given a bounded linear operator $T\in\cB(\cH)$ for some Hilbert space $\cH$, we denote by 
$\rho(T):=\{\lambda\in\CC\mid \lambda I-T\mbox{ has a bounded inverse}\}$ the
\emph{resolvent set} of $T$.
The next result is Theorem~4.1 in \cite{ConstantinescuGheondea1}.

\begin{theorem}\label{t:unic} Given $K\in\fK^h(X;\bH)$, the following assertions are equivalent.
\begin{itemize}
\item[(i)] $K$ has a unique $\bH$-valued Krein space minimal linearisation, modulo unitary equivalence.
\item[(ii)] For any $L\in\fK^+(X;\bH)$ such that $-L\leq K\leq L$, there exists $\epsilon>0$ such that
either $(-\epsilon,0)\subset \rho(G_K)$ or $(0,\epsilon)\subset \rho(G_K)$.
\end{itemize}
\end{theorem}

Other criteria of uniqueness of minimal linearisations of $\bH$-operator valued Hermitian kernels are available, see, 
for example, Proposition~4.3 in \cite{ConstantinescuGheondea1}. A stronger result, saying that, in the case of nonuniqueness, 
infinitely many nonunitary minimal linearisations exist, can be derived from the results of 
B.~\'Curgus and H.~Langer in \cite{CurgusLanger}.

\subsection{Reproducing Kernel Krein Spaces}\label{ss:rkk}
Consider $X$ a nonempty set and let $X\ast \bH$ be a bundle of Hilbert spaces over the field $\FF$. 
Let $K\in\fK^h(X;\bH)$ and, as usually, let $K_x:=K(\cdot,x)$, for all $x\in X$. 
An \emph{$\bH$-valued reproducing kernel Krein space} with kernel $K$ is, by definition, a Krein space 
$(\cR;[\cdot,\cdot])$ subject to the following conditions.
\begin{itemize}
\item[(rkk1)] $\cR\subseteq \cF_\FF(X;\bH)$, with all algebraic operations.
\item[(rkk2)] For any $x\in X$ and $h\in \cH_x$ we have $K_xh\in \cR$.
\item[(rkk3)] For any $f\in\cR$, $x\in X$, and $h\in\cH_x$, we have $\langle f(x),h\rangle_{\cH_x}=[ f,K_xh]$.
\end{itemize}

Property (rkk3) is called the \emph{reproducing property}. As a consequence, the following minimality condition holds as well.
\begin{itemize}
\item[(rkk4)] $\cR$ equals the closed linear span of $\{K_xh\mid x\in X,\ h\in\cH_x\}$.
\end{itemize}

Concerning existence of $\bH$-valued reproducing kernel Krein spaces associated to a given kernel $K\in\fK(X;\bH)$, 
the next result is Theorem~5.1, Corollary~5.2, and Corollary~5.3 in \cite{ConstantinescuGheondea1}.

\begin{theorem}\label{t:rkks}
Let $K\in\fK^h(X;\bH)$. Then, the mapping
\begin{equation*}
(\cK;V)\mapsto \cR(\cK;V):=\{(V(x)^\sharp f)_{x\in X} \mid f\in \cK\} 
\end{equation*}
with inner product on $\cR(\cK;V)$ given by
\begin{equation*}[(V(x)^\sharp f)_{x\in X},(V(x)^\sharp g)_{x\in X}]:= [f,g]_\cK,\quad f,g\in \cK,
\end{equation*}
maps the class of all $\bH$-valued minimal linearisations $(\cK;V)$ of $K$ onto the class of all $\bH$-valued reproducing kernel 
Krein spaces $(\cR,[\cdot,\cdot])$ with reproducing kernel $K$.

In addition, in this correspondence, two $\bH$-valued minimal linearisations $(\cK;V)$ and $(\cH;U)$ of $K$ are unitarily 
equivalent, if and only if the $\bH$-valued reproducing kernel Krein spaces $\cR(\cK;V)$ and $\cR(\cH;U)$ coincide.
\end{theorem}

As a consequence of this result and Theorem~\ref{t:unic}, one can get the characterisation of uniqueness of $\bH$-valued 
reproducing kernel Krein spaces induced by a given kernel $K\in\fK^h(X;\bH)$. 

\begin{theorem}\label{t:unick} Let $K\in\fK^h(X;\bH)$ be such that it has $\bH$-valued Krein space linearisations, equivalently, 
it has $\bH$-valued reproducing kernel Krein spaces. The following assertions are equivalent.
\begin{itemize}
\item[(i)] $K$ has unique $\bH$-valued reproducing kernel Krein space, modulo unitary equivalence.
\item[(ii)] For any $L\in\fK^+(X;\bH)$ such that $-L\leq K\leq L$, there exists $\epsilon>0$ such that, with respect to notation as 
in \eqref{e:gram}, either $(-\epsilon,0)\subset \rho(G_K)$ or $(0,\epsilon)\subset \rho(G_K)$.
\end{itemize}
\end{theorem}


\subsection{Generalised $\Gamma$-Invariant Krein Space Linearisations}\label{ss:ihk}
In this subsection we consider $\bH$-valued Hermitian kernels that are invariant under actions of $*$-semigroupoids with unit.
The results will be generalisations of Theorem~\ref{t:invariantu} and Theorem~\ref{t:invariantul} 
to the case when representations are obtained
on $*$-semigroupoids modelled by operators, bounded or unbounded, on bundles of Krein spaces. To this end, we start by 
explaining the concepts of these $*$-semigroupoids. The difference with respect to Example~\ref{ex:segah} and 
Example~\ref{ex:semuh} will be performed by considering the Krein space involution instead of Hilbert space involution.

\begin{example}\label{ex:segahaba} (\emph{$*$-Semigroupoids Modelled by Bounded Operators on Krein Spaces})
Let $S$ be a nonempty set and $\bK=\{\cK_s\}_{s\in S}$ be a \emph{bundle of Krein spaces}
over $\FF$ for some nonempty set $S$.  
Since strong topologies on Krein spaces are provided by underlying Hilbert space structures 
and are unique, the organisation of the semigroupoid is similar to that explained in Example~\ref{ex:segah}. So, let 
\begin{equation}\label{e:gamabehab}
\Gamma_{\bK}=\bigsqcup_{s,t\in S} \cB(\cK_s,\cK_t),
\end{equation}
where $\cB(\cK_s,\cK_t)$ denotes the vector space of all bounded linear operators 
$T\colon \cK_s\ra\cK_t$. Then, the maps $d$ and $c$ are defined correspondingly and, similarly, the product on 
$\Gamma_{\bK}$. Thus, $(\Gamma_{\bK},S,d_{\bK},c_{\bK},\cdot)$ is a semigroupoid.

The semigroupoid $\Gamma_{\bK}$ has the unit $\epsilon\colon S\ra\Gamma_{\bK}$ defined by 
$\epsilon_s:=I_{\cK_s}$, where $I_{\cK_s}$ denotes the identity operator on $\cK_s$, for all $s\in S$.
It also has a natural involution: $\Gamma_{\bK}\ni T\mapsto T^\sharp\in\Gamma_{\bK}$ where, 
if $T\in\cB(\cK_s,\cK_t)$, by $T^\sharp\in\cB(\cK_t,\cK_s)$ we denote the Krein space adjoint operator.
In addition, the semigroupoid $\Gamma_{\bK}$ has a natural left action on $S$ defined by 
$T\cdot s:=t$, where $T\in\cB(\cK_s,\cK_t)$.
\end{example}

\begin{example}\label{ex:semuha}
 (\emph{$*$-Semigroupoids Modelled by Unbounded Operators in Krein Spaces.})
\label{ex:semesek} 
Let us assume, with notation as in Example~\ref{ex:semes},
that, for each $s\in S$, the vector space 
$\cD_s$ is a dense subspace of a Krein space $\cK_s$, and let
\begin{equation}\label{e:gabedek}
\Gamma_{\bK;\bD}:=\bigsqcup_{s,t\in S} \cL^\sharp(\cD_s,\cD_t),
\end{equation}
where, for each $s,t\in S$, we denote by $\cL^\sharp(\cD_s,\cD_t)$ the vector space of all
linear operators $T\colon \cD_s(\subseteq \cK_s)\ra\cK_t$, subject to the following assumptions.
\begin{itemize}
\item[(i)] $T\cD_s\subseteq \cD_t$.
\item[(ii)] $ \cD_t\subseteq \dom(T^\sharp)$ and $T^\sharp\cD_t\subseteq\cD_s$.
\end{itemize}
Here, by $T^\sharp\colon\dom(T^\sharp)(\subseteq \cK_t)\ra\cK_s$ we understand 
the adjoint operator (possibly unbounded) defined in the usual sense, see Subsection~\ref{ss:kstlo},
\begin{equation*}
\dom(T^\sharp):=\{k\in \cK_t\mid \cD_s\ni h\mapsto [Th,k]_{\cK_t}\mbox{ is bounded}\},
\end{equation*}
and
\begin{equation*}
[Th,k]_{\cK_t}=[h,T^\sharp k]_{\cK_s},\quad h\in\cD_s,\quad k\in\dom(T^\sharp).
\end{equation*}
Note that, in this way, any operator $T$ in $\cL^\sharp(\cD_s,\cD_t)$ is closable. The involution on $\Gamma_{\bK,\bD}$ is
defined by $\cL^\sharp(\cD_s,\cD_t)\ni T\mapsto T^\sharp|_{\cD_t}$,  for all $s,t\in S$.

Then, it is easy to see that $\Gamma_{\bK;\bD}\ni T\mapsto T^\sharp\in\Gamma_{\bK;\bD}$ is an involution
which makes $\Gamma_{\bK;\bD}$ a $*$-semigroupoid with unit. As in the previous example, 
there is a natural left action of $\Gamma_{\bK;\bD}$ onto $S$.
\end{example}

Let $K\colon X\times X\ra\Gamma_{\bH}$ be a kernel, where $\bH=\{\cH_x\mid x\in X\}$ is a bundle of Hilbert 
spaces over the field $\FF$, where the $*$-semigroupoid with
$\Gamma_{\bH}$ is defined in 
Example~\ref{ex:segah}. In particular, $K(x,y)\in \cB(\cH_y,\cH_x)$ for all $x,y\in X$. Recall the definitions of partially Hermitian 
and partially positive semidefinite, with respect to a partition of $X$, as in Subsection~\ref{ss:ipsdk}. We can have more general 
definitions, in terms of an arbitrary partition $P=\{X_s\}_{s\in S}$, where $S$ is a nonempty set, not necessarily the set of 
symbols of the semigroupoid $\Gamma$. We denote by $\fK^h_P(X;\bH)$ the real vector space of all $\bH$-operator valued
kernels which are partially Hermitian with respect to the partition $P$ and by $\fK^+_P(X;\bH)$ the convex cone of all 
$\bH$-operator valued kernels that are partially positive semidefinite with respect to the partition $P$.
If $K,L\in \fK^h_P(X;\bH)$, we say that $L$ \emph{partially dominates} $K$ with respect to the partition $P$, and we write
$K\leq_{\mathrm{P}}L$, if $L-K\in \fK^+_P(X;\bH)$. In particular, if $L\in \fK^h_P(X;\bH)$, we can equivalently write 
$L\geq_{\mathrm{P}}0$.

The next result is a localised version of Theorem~\ref{t:ksl} and its proof is a direct consequence of it, hence we omit it. We also 
include a characterisation in terms of reproducing kernel Krein spaces, via Theorem~\ref{t:rkks}.

\begin{theorem}\label{t:kslp}
Given a partition $P=\{X_s\}_{s\in S}$ of $X$ and $K\in\fK^h_P(X;\bH)$, the following assertions are equivalent.
\begin{itemize}
\item[(1)] There exists $L\in\fK^+_P(X;\bH)$ such that $-L\leq_{\mathrm{P}} K\leq_{\mathrm{P}} L$.
\item[(1)$^\prime$] There exists $L\in\fK^+_P(X;\bH)$ such that $K\leq_{\mathrm{P}} L$.
\item[(2)] There exists $L\in\fK^+_P(X;\bH)$ such that, for each $s\in S$, we have
\begin{equation}\label{e:fegald}
|[f,g]_K|\leq [f,f]_L^{1/2} [g,g]_L^{1/2},\quad f,g\in \cF_\FF^0(X_s;\bH).
\end{equation}
\item[(3)] $K=K_1-K_2$, for some $K_1,K_2\in \fK^+_P(X;\bH)$.
\item[(3)$^\prime$] $K=K_1-K_2$, for some disjoint $K_1,K_2\in \fK^+_P(X;\bH)$.
\item[(4)] There exists a pair $(\bK;\bV)$ with the following properties.
\begin{itemize}
\item[(LL1)] $\bK=\{\cK_s\}_{s\in S}$ is a bundle of Krein spaces.
\item[(LL2)] $\bV=\{V_x\}_{x\in X}$ is a bundle of bounded linear operators, 
$V_x\colon \cH_x\ra \cK_{a(x)}$ for each $x\in X$, such that, for all $s\in S$ and all $x,y\in X_s$, we have
 $K(x,y)=V_x^*J_{s}V_y$, where $J_{s}$ is a fixed, but otherwise arbitrary, fundamental symmetry of $\cK_{s}$,
 equivalently,
 \begin{equation}
 \langle K(x,y)k,h\rangle_{\cH_x}=[V_y k,V_x h]_{\cK_s},\quad h\in\cH_x,\ k\in\cH_y.
 \end{equation}
\item[(LL3)] For each $s\in S$, the Hilbert space $\cK_s$ is the closed span of $\{V_x \cH_x\mid x\in X,\ a(x)=s\}$.
\end{itemize}
\item[(5)] There exists a Krein space 
bundle $\bR=\{\cR_{K^s}\}_{s\in S}$, with $K^s=K|_{X_s\times X_s}$ and $\cR_{K^s}$ denoting 
a reproducing kernel $\bH_s$-valued Krein space with $\bH_s$-operator valued reproducing kernel $K^s$, where 
$\bH_s:=\{\cH_x\}_{x\in X_s}$, for all $s\in S$.
\end{itemize} 
\end{theorem} 

A pair $(\bK;\bV)$ with properties (LL1) and (LL2) is called a  \emph{Krein space partial linearisation} 
of $K$ with respect to the partition $P$. If, in addition, property (LL3) holds, then we call it \emph{minimal}.

We consider now the setting as in Subsection~\ref{ss:ipsdk}, so
$X$ is a nonempty set and $(a,\cdot)$ is a left action of the $*$-semigroupoid $\Gamma$ on $X$, with $a\colon X\ra S$ the anchor map and 
$d\colon \Gamma\ra S$ and $c\colon \Gamma\ra S$ denote the domain map and the codomain map, 
respectively. Since the anchor map $a$ is surjective, for each $s\in S$ the set $X_s$ defined in \eqref{e:xes} is nonempty
and we have a partition of $X$ as in \eqref{e:part}, that we denote by $P_a$.

With notations and assumptions as before, 
if $x\in X$, let $\Gamma\cdot x:=\{\alpha\cdot x\mid \alpha\in \Gamma_{a(x)}\}$ be the $\Gamma$-orbit of $x$. We assume, in
addition, that condition \eqref{e:orbit} holds, equivalently, this means that the Hilbert space bundle 
$\{\cH_y\}_{y\in \Gamma\cdot x}$ is trivial, for all $x\in X$.
Recall that, the kernel $K$ is called \emph{invariant} under the action of $\Gamma$ on $X$, or simply 
\emph{$\Gamma$-invariant}, if condition \eqref{e:invariant} holds.

Our first goal is to show that $\Gamma$-invariant kernels yield certain generalised
$*$-representations of the $*$-semigroupoid $\Gamma$ on certain bundles of reproducing kernel Krein spaces, 
as a generalisation of Theorem~\ref{t:invariantu}. 

\begin{theorem}\label{t:invariantuk} With notation as before,
let $K\in \fK^h_{P_a}(X;\bH)$. The following assertions are equivalent.

\nr{1} $K$ is a
$\Gamma$-invariant kernel, in the sense that both \eqref{e:orbit} and \eqref{e:invariant} hold, such that there exists
$L\in  \fK^+_{P_a}(X;\bH)$ with the property that $K\leq_{\mathrm{P_a}} L$.

\nr{2} There exists a bundle $\bR=\{\cR_{K^s}\}_{s\in S}$, with $K^s=K|_{X_s\times X_s}$ and $\cR_{K^s}$ denoting 
a reproducing kernel $\bH_s$-valued Krein space with $\bH_s$-operator valued kernel $K^s$, where 
$\bH_s:=\{\cH_x\}_{x\in X_s}$, for all $s\in S$, and, letting
$\bD=\{\cD_s\}_{s\in S}$ with $\cD_s=\lin\{K^s_x\mid x\in X_s\}$, for all $s\in S$,
there exists a generalised $*$-representation $(\mathrm{Id}_S,\Phi)$, in the sense of 
Definition~\ref{d:urep}, where $\Phi\colon \Gamma\ra \Gamma_{\bR,\bD}$, where the $*$-semigroupoid $\Gamma_{\bR,\bD}$ is
defined as in Example~\ref{ex:semuha}, such that, for any $\alpha\in \Gamma$ and any $x\in X_{d(\alpha)}$, we have 
\begin{equation}\label{e:fivak}
K^{c(\alpha)}_{\alpha \cdot x}=\Phi(\alpha) K^{d(\alpha)}_x.
\end{equation}

\nr{3} There exists a quadruple $(\bK;\bD;\bV;\Phi)$ with the following properties.
\begin{itemize}
\item[(GIKL1)] The pair $(\bK;\bV)$ is a generalised Krein space minimal partial linearisation of $K$.
\item[(GIKL2)] $(\mathrm{Id}_S;\Phi)$ is a generalised $*$-representation of $\Gamma$ on $\Gamma_{\bK,\bD}$
and $\Phi(\alpha) V_x=V_{\alpha\cdot x}$ for all $x\in X$ and all $\alpha\in\Gamma_{a(x)}$.
\end{itemize}
\end{theorem}

\begin{proof} (1)$\Ra$(2). We use Theorem~\ref{t:kslp} and get a Krein space bundle 
$\bR=\{\cR_{K^s}\}_{s\in S}$, with $K^s=K|_{X_s\times X_s}$ and $\cR_{K^s}$ is 
a reproducing kernel $\bH_s$-valued Krein space with $\bH_s$-operator valued kernel $K^s$, where 
$\bH_s:=\{\cH_x\}_{x\in X_s}$, for all $s\in S$. Then we practically proceed as in the proof of 
Theorem~\ref{t:invariantu}, observing that the positive semidefiniteness does not play any significant role, with 
the only difference that, instead of Hilbert spaces we get Krein spaces. Because of that, we will be short.

Let $\alpha\in \Gamma$. We define the linear operator $\Phi(\alpha)\colon \cD_{d(\alpha)}\ra \cD_{c(\alpha)}$ by
\begin{equation}\label{e:phia}
\Phi(\alpha)f:=(f_{\alpha^*\cdot x})_{x\in X_{c(\alpha)}},\quad f\in \cD_{d(\alpha)}.
\end{equation}
Inspecting the proof in Theorem~\ref{t:invariantu}, it follows that $(\mathrm{Id}_S,\Phi)$ is a generalised $*$-representation of 
the $*$-semigroupoid $\Gamma$ on $\Gamma_{\bR,\bD}$ in the sense of Definition~\ref{d:urep} and
we show that $\Phi$ has the property \eqref{e:fivak}.

(2)$\Ra$(3). This is actually a consequence of Theorem~\ref{t:rkks}.

(3)$\Ra$(1). We prove that $K$ is $\Gamma$-invariant. Indeed, let $x,y\in X$ and $\alpha \in\Gamma_{a(x)}^{a(y)}$ be  
arbitrary. Then, for any $h\in \cH_x$ and $k\in \cH_y$, we have
\begin{align*}
\langle K(\alpha\cdot x,y)k,h\rangle_{\cH_x} & = [V(y)k,V(\alpha\cdot x)h]_\cK= [V(y)k,\Phi(\alpha)V(x)h]_\cK \\
& = [\Phi(\alpha^*)V(y)k,V(x)h]_\cK = [V(\alpha^*\cdot y)k,V(x)h]_\cK \\
& = \langle V(x)^* V(\alpha^*y)k,h\rangle_{\cH_x} = \langle K(x,\alpha^*\cdot y)k,h]_\cK.
\end{align*}

We prove that $K$ is partially Hermitian with respect to the partition $P_a$. To see this, let $s\in S$ and $x,y\in X_s$ be arbitrary.
Then,
\begin{equation*}
K(x,y)^*=(V(x)^* J_s V(y))^*=V(y)^* J_s V(x)=K(y,x).
\end{equation*}

Finally, for each $s\in S$, let $L^s(x,y):=V(x)^*V(y)$, $x,y\in X_s$. Then define the $\bH$-operator valued kernel $L$ by
\begin{equation*}
L(x,y):=\begin{cases} L^s(x,y), & x,y\in X_s, \\
0,& \mbox{in the opposite case}.
\end{cases}
\end{equation*}
Then, fix $s\in S$ and let $x_1,\ldots, x_n\in X_s$ be arbitrary. 
We consider the operator block $n\times n$ matrix with $J_s$ on the 
diagonal and zero elsewhere, which is selfadjoint and dominated by the identity matrix, that is,
\begin{equation*}
\begin{bmatrix} J_s & 0 & \cdots & 0 \\ 0 & J_s & \cdots & 0 \\ \vdots & \vdots & \cdots & \vdots \\ 0 & 0 & \cdots & J_s
\end{bmatrix} \leq 
\begin{bmatrix} I_s & 0 & \cdots & 0 \\ 0 & I_s & \cdots & 0 \\ \vdots & \vdots & \cdots & \vdots \\ 0 & 0 & \cdots & I_s
\end{bmatrix},
\end{equation*}
where $I_s$ denotes the identity operator on the Hilbert space $(\cK_s;\langle \cdot,\cdot\rangle_{J_s})$. By multiplying both 
sides of the inequality to the 
left with the operator $\begin{bmatrix} V(x_1) & V(x_2) & \cdots & V(x_n)\end{bmatrix}$ and to the right with its adjoint, we 
preserve the inequality. Then, observing that,
\begin{equation*}
[K(x_i,x_j)]_{i,j=1}^n\! =\! \begin{bmatrix} V(x_1)^*J_s V(x_1) & V(x_1)^*J_sV(x_2) & \cdots & V(x_1)^*J_sV(x_n) \\ V(x_2)^*J_sV(x_1) & V(x_2)^*J_s 
V(x_2) & \cdots & V(x_2)^*J_sV(x_n) \\ \vdots & \vdots & \cdots & \vdots \\ V(x_n)^*J_s V(x_1) & V(x_n)^*J_s V(x_2) & \cdots & 
V(x_n)^*J_sV(x_n)
\end{bmatrix}
\end{equation*}
and that
\begin{equation*}
[L(x_i,x_j)]_{i,j=1}^n =
\begin{bmatrix} V(x_1)^* V(x_1) & V(x_1)^*V(x_2) & \cdots & V(x_1)^*V(x_n) \\ V(x_2)^*V(x_1) & V(x_2)^* 
V(x_2) & \cdots & V(x_2)^*V(x_n) \\ \vdots & \vdots & \cdots & \vdots \\ V(x_n)^* V(x_1) & V(x_n)^* V(x_2) & \cdots & 
V(x_n)^*V(x_n),
\end{bmatrix}
\end{equation*}
this proves that $[K(x_i,x_j)]_{i,j=1}^n\leq [L(x_i,x_j)]_{i,j=1}^n$ and, hence, $K\leq_{\mathrm{P_a}} L$.
\end{proof}

\subsection{$\Gamma$-Invariant Krein Space Linearisations}\label{ss:iksl}
Finally, we are interested in characterising those kernels  $K\in \fK^h_{P_a}(X;\bH)$ for which a left action of a $*$-semigroupoid 
$\Gamma$ on $X$ yields $*$-representations of $\Gamma$ with bounded operators, either on partial linearisations or on 
bundles of reproducing kernel Krein spaces of $K$. For this, we have to recall the concept of kernels 
$L\in \fK^+_{P_a}(X;\bH)$ 
that are of bounded shift type, as given in Subsection~\ref{ss:ipsdk}. Also, given 
$L_1,L_2\in \fK^+_{P_a}(X;\bH)$, we say that
$L_1$ and $L_2$ are disjoint if, whenever $L\in \fK^+_{P_a}(X;\bH)$ is such that 
$L\leq_{\mathrm{P_a}}L_1$ and $L\leq_{\mathrm{P_a}}L_2$, it follows that $L^s:=L|_{X_s\times X_s}$ is the 
zero kernel, for all $s\in S$.

\begin{theorem}\label{t:invariantukbp}
Let $K\in \fK^h_{P_a}(X;\bH)$. The following assertions are equivalent.

\nr{1} There exists $L\in\fK^+_P(X;\bH)$ of bounded shift type and
such that $-L\leq_{\mathrm{P_a}} K\leq_{\mathrm{P_a}} L$.

\nr{2} There exists $L\in\fK^+_P(X;\bH)$ of bounded shift type and 
such that, for each $s\in S$, we have
\begin{equation}\label{e:fegalt}
|[f,g]_K|\leq [f,f]_L^{1/2} [g,g]_L^{1/2},\quad f,g\in \cF_\FF^0(X_s;\bH_s).
\end{equation}

\nr{3} $K=K_1-K_2$ for some $K_1,K_2\in \fK^+_P(X;\bH)$, both of them of bounded shift type.

\nr{3}$^\prime$ $K=K_+-K_-$ for some disjoint $K_\pm\in \fK^+_P(X;\bH)$, both of them of 
bounded shift type.

\nr{4}  There exists a triple $(\bK;\bV;\Psi)$ with the following properties.
\begin{itemize}
\item[(IKL1)] The pair $(\bK;\bV)$ is a  minimal partial Krein space linearisation of $K$.
\item[(IKL2)] $(\mathrm{Id}_S;\Psi)$ is a $*$-representation of $\Gamma$ on $\Gamma_{\bK}$
such that, $\Psi(\alpha) V_x=V_{\alpha\cdot x}$, for all $x\in X$ and all $\alpha\in\Gamma_{a(x)}$.
\end{itemize}

\nr{5} There exists a Krein space bundle $\bR=\{\cR_{K^s}\}_{s\in S}$, with $K^s=K|_{X_s\times X_s}$ and $\cR_{K^s}$ is 
a reproducing kernel $\bH_s$-valued Krein space with $\bH_s$-operator valued reproducing kernel $K^s$, 
where $\bH_s:=\{\cH_x\}_{x\in X_s}$, for all $s\in S$, and
there exists a  $*$-representation $(\mathrm{Id}_S,\Phi)$, in the sense of 
Definition~\ref{d:serep}, where $\Phi\colon \Gamma\ra \Gamma_{\bR}$ and the $*$-semigroupoid $\Gamma_{\bR}$ is
defined as in Example~\ref{ex:segahaba}, such that, for any $\alpha\in \Gamma$ and any $x\in X_{d(\alpha)}$, we have 
\begin{equation}\label{e:fivaka}
K^{c(\alpha)}_{\alpha \cdot x}=\Phi(\alpha) K^{d(\alpha)}_x.
\end{equation}
\end{theorem}

\begin{proof} (1)$\Leftrightarrow$(2). This is essentially obtained as in Proposition~38 in \cite{Schwartz}, see also Theorem~3.1 
in \cite{ConstantinescuGheondea1}, and we do not repeat the details.

(3)$^\prime\Ra$(3). Clear.

(3)$\Ra$(1). Clear.

(4)$\Leftrightarrow$(5). Any reproducing kernel Krein space is a minimal Krein space linearisation and the converse is obtained 
as in Theorem~\ref{t:rkks}.

(2)$\Ra$(4). Let $L\in\fK^+_P(X;\bH)$ be of bounded shift type and 
such that, for each $s\in S$, the inequality \eqref{e:fegalt} holds. For each $s\in S$, we denote by $\cG_s$ the Hilbert space 
obtained as a quotient-completion of the vector space $\cF_\FF^0(X_s;\bH_s)$ with respect to the positive semidefinite
inner product $\langle\cdot,\cdot\rangle_{L^s}$ defined as 
in \eqref{e:ipk}. Considering the convolution operator $C_{L^s}$ defined as in \eqref{e:conv}, we have
\begin{equation*}
\langle C_{L^s}f,g\rangle_0=\langle f,g\rangle_{L^s},\quad f,g\in \cF_\FF^0(X_s;\bH_s).
\end{equation*} 
With notation as in \eqref{e:psia}, \eqref{e:psias}, and \eqref{e:cedes}, we consider the semigroupoid morphism 
$(\mathrm{Id}_S,\Psi)$, see Lemma~\ref{l:psi}. Since $L$ is of bounded shift type, for each $\alpha\in\Gamma$, in view of
Lemma~\ref{l:bst}, the operator $\Psi(\alpha)$ uniquely induces a bounded linear operator between the Hilbert spaces
$\cG_{d(\alpha)}$ and $\cG_{c(\alpha)}$, denoted yet by $\Psi(\alpha)$. In view of the multiplicativity property of $\Psi$, see
Lemma~\ref{l:psi}, in this way we obtain a semigroupoid representation $\Psi\colon \Gamma\ra \Gamma_{\bG}$, where the 
semigroupoid $\Gamma_{\bG}$ is defined as in Example~\ref{ex:segah}.

Since $K$ is $\Gamma$-invariant and observing that the fact that kernel $K$ in \eqref{e:lapsi} is partially positive semidefinite 
does not play any role, what actually matters is that it is partially Hermitian, exactly in 
the same fashion we can prove that, for any $\alpha\in\Gamma$,  $f\in \cF^0_\FF(X_{d(\alpha)};\bH_{d(\alpha)})$, and
$g\in \cF_\FF^0(X_{c(\alpha)};\bH_{c(\alpha)})$, we have
\begin{equation}\label{e:sepsi}
[\Psi(\alpha)f,g]_K=[f,\Psi(\alpha^*)g]_K.
\end{equation}

In view of the inequality \eqref{e:fegalt},  for each $s\in S$,
the indefinite inner product $[\cdot,\cdot]_{K^s}$ factors and extends to the Hilbert space 
$\cG_s$, and the same inequality holds, hence there exists uniquely a selfadjoint contraction 
$G_{K^s}\in\cB(\cG_s)$, the Gram operator of $[\cdot,\cdot]_{K^s}$ with respect to 
$\langle\cdot,\cdot\rangle_{L^s}$. Using again
the quotient completion method of inducing a Krein space from a selfadjoint bounded operator, 
as described in Example~\ref{ex:kas}, we produce the Krein space $\cK_s:=\cK_{G_{K^s}}$.
From \eqref{e:sepsi}, it follows that, for each $\alpha\in\Gamma$, the operators $\Psi(\alpha)$ and 
$\Psi(\alpha^*)$ are adjoints of each other, with respect to the indefinite inner products 
$[\cdot,\cdot]_{K^{d(\alpha)}}$ and $[\cdot,\cdot]_{K^{c(\alpha)}}$, 
hence Theorem~\ref{t:abetul} is applicable and, consequently, both $\Psi(\alpha)$ and $\Psi(\alpha^*)$ 
can be lifted as bounded linear operators, 
\begin{equation}
\Psi(\alpha)\colon \cK_{d(\alpha)}\ra \cK_{c(\alpha)},\quad \Psi(\alpha^*)\colon \cK_{c(\alpha)}\ra \cK_{d(\alpha)},
\end{equation}
such that
\begin{equation}
[\Psi(\alpha)f,g]_{K}=[f,\Psi(\alpha^*)g]_K,\quad f\in \cK_{d(\alpha)},\ g\in \cK_{c(\alpha)},
\end{equation}
that is, $(\mathrm{Id}_S;\Psi)$ can be viewed as a $*$-representation of $\Gamma$ on the 
$*$-semigroupoid $\Gamma_{\bK}$, defined as in Example~\ref{ex:segahaba}.

For each $x\in X$, the operator $V(x)\colon \cH_x\ra \cK_{a(x)}$ is defined in the natural fashion
\begin{equation}\label{e:vexah}
V(x)h:= \delta_x h+\cN_{K^{a(x)}},\quad h\in\cH_x.
\end{equation}
We observe that, for each $x\in X$, considering the map $\cH_x\ni h\mapsto 
V_L(x)h:=\delta_x h+\cN_{L^{a(x)}}\in 
\cG_{a(x)}$ and the canonical operator $\Pi_{G_{L^{a(x)}}}\colon \cG_{a(x)}\ra \cK_{a(x)}=\cK_{G_{L^{a(x)}}}$, we have
\begin{equation}\label{e:vexapi}
V(x)= \Pi_{G_{L^{a(x)}}} \circ V_L(x).
\end{equation}
Indeed, by \eqref{e:fegalt}, we have $\cN_{L^{a(x)}}\subseteq \cN_{K^{a(x)}}$ and, consequently,
\begin{equation*}
\cN_{K^{a(x)}}=\cN_{L^{a(x)}}+\ker(G_{L^{a(x)}}).
\end{equation*}
Then, for each $h\in \cH_x$ we have
\begin{align*}
\bigl(\Pi_{G_{L^{a(x)}}} \circ V_L(x)\bigr)h & =  \Pi_{G_{L^{a(x)}}} (\delta_x h+ \cN_{L^{a(x)}})\\
& =\delta_x h+\bigl(\cN_{L^{a(x)}}+\ker(G_{L^{a(x)}})\bigr)=\delta_x h+ \cN_{K^{a(x)}},
\end{align*}
hence \eqref{e:vexapi} is proven and, consequently, $V(x)$ is a bounded linear operator. 
Then, for each $x,y\in X$, 
such that $a(x)=a(y)=s$, and each $h\in\cH_x$ and $k\in\cH_y$, we have
\begin{align*}
[V(y)k,V(x)h]_{\cK_s} & =[\delta_y k+\cN_{K^s},\delta_x h+\cN_{K^s}]_{K^s}=\langle C_{K^s}\delta_y k,\delta_x h\rangle_0 \\
& = \sum_{z\in X_s}\langle (C_{K^s}\delta_y k)(z),\delta_x(z)h\rangle_{\cH_x}=\langle (C_{K^s}\delta_yk)(x),h\rangle_{\cH_x} \\
& = \langle \sum_{z\in X_s}\delta_y(z)K^s(x,z)k,h\rangle_{\cH_x} = \langle K(x,y)k,h\rangle_{\cH_x}.
\end{align*}
All these prove that the triple $(\bK;\bV;\Psi)$ is a $\Gamma$-invariant minimal Krein space linearisation of $K$.

(4)$\Ra$(3)$^\prime$. Let $(\bK;\bV;\Psi)$ be a triple subject to the conditions (IKL1) and (IKL2). For each $s\in S$, let $J_s$ be 
a fundamental symmetry on $\cK_s$ and consider the Hilbert space inner product $\langle \cdot,\cdot\rangle_{J_s}$ induced by 
$J_s$, that is, 
\begin{equation}
\langle h,k\rangle_{J_s}=[J_s h,k]_{\cK_s},\quad h,k\in \cK_s.
\end{equation}
Let $J_s=J_s^+-J_s^-$ be the Jordan decomposition of $J_s$, that is, $J_s^\pm$ is the spectral projection of $\cK_s$ onto 
$\ker(I_{\cK_s}\mp J_s)$, respectively. Then $I_{\cK_s}=J_s^++J_s^-$ and $J_s^+J_s^-=0$. For each $x,y\in X$, define
\begin{equation}\label{e:kpm}
K_\pm(x,y)= \begin{cases} V^*(x)J_s^\pm V(y), & a(x)=a(y)=s,\\
0,& \mbox{in the opposite case}.
\end{cases}
\end{equation}
Clearly, $K_\pm$ are positive semidefinite and the same argument as the one used in proof of Theorem~\ref{t:ksl} shows that
they are disjoint. It remains to prove that the $\bH$-operator valued kernel $K_++K_-$ is of bounded shift type. To see this, let
$\alpha\in\Gamma$ be arbitrary. Then,
\begin{align*}
\sum_{x,y\in X_{d(\alpha)}} & \!\!\!\!
\langle (K_+(\alpha\cdot x, \alpha\cdot y)  +K_-(\alpha\cdot x,\alpha\cdot y))g_y,g_x\rangle_{\cH_x} \\
&  \!\!\!\!\!\!\! = \!\!\!\!\!\! \sum_{x,y\in X_{d(\alpha)}} \!\!\!\!\! \langle (V(\alpha\cdot x)^*J_s^+V(\alpha\cdot y)
+V(\alpha\cdot x)^*J_s^-V(\alpha\cdot y))g_y,g_x\rangle_{\cH_x} \\
&  \!\!\!\!\!\!\! = \!\!\!\!\!\!\!\!\sum_{x,y\in X_{d(\alpha)}}  \!\!\!\!\! \langle (V(x)^*\Psi(\alpha)^*J_s^+\Psi(\alpha)V(y)
+V(x)^*\Psi(\alpha)^*J_s^-\Psi(\alpha)V(y))g_y,g_x\rangle_{\cH_x} \\
&  = \bigl\langle J_s^+ \Psi(\alpha)\sum_{y\in X_{d(\alpha)}} V(y)g_y,J_s^+ \Psi(\alpha)\sum_{x\in X_{d(\alpha)}} V(x)g_x\bigr\rangle_{J_s} 
\\
& \phantom{= \sum_{x,y\in X_{d(\alpha)}}} +
\bigl\langle J_s^- \Psi(\alpha) \sum_{y\in X_{d(\alpha)}} V(y)g_y,J_s^- \Psi(\alpha)\sum_{x\in X_{d(\alpha)}} V(x)g_x\bigr\rangle_{J_s}\\
\intertext{which, taking into account that $J_s^\pm$ are orthogonal projections, $J_s^+J_s^-=0$, and $J_s^++J_s^-=I_s$, equals} 
& = \bigl\langle \Psi(\alpha) \sum_{y\in X_{d(\alpha)}}V(y)g_y,\Psi(\alpha)\sum_{x\in X_{d(\alpha)}}V(x)g_x\bigr\rangle_{J_s} \\
\intertext{which, by the same argument as before, equals}
& \leq \|\Psi(\alpha)\|^2   \bigl\langle \sum_{y\in X_{d(\alpha)}}V(y)g_y,\sum_{x\in X_{d(\alpha)}}V(x)g_x\bigr\rangle_{J_s}\\
& =  \|\Psi(\alpha)\|^2 \biggl(  \bigr\langle J_s^+\sum_{y\in X_{d(\alpha)}}V(y)g_y,J_s^+\sum_{x\in X_{d(\alpha)}}V(x)g_x\bigr\rangle_{J_s}\biggr. \\
& \phantom{= \sum_{x,y\in X_{d(\alpha)}}} + \biggl. \bigl\langle J_s^-\sum_{y\in X_{d(\alpha)}}V(y)g_y,J_s^-\sum_{x\in X_{d(\alpha)}}V(x)g_x\bigr\rangle_{J_s}\biggr) \\
& = \|\Psi(\alpha)\|^2\biggl( \sum_{x,y\in X_{d(\alpha)}}\!\!\! \langle \bigl(V(x)^*J_s^+V(y)+V(x)^*J_s^-V(y)g_y,g_x\rangle_{\cH_x}
\biggr)  \\
& = \|\Psi(\alpha)\|^2 \sum_{x,y\in X_{d(\alpha)}} \langle \bigl(K_+(x,y) +K_-(x,y)\bigr)g_y,g_x\rangle_{\cH_x}.\qedhere
\end{align*}
\end{proof}

\begin{remark} According to Corollary~\ref{c:invsem}, if $\Gamma$ is an inverse semigroupoid, then in assertions (1) through (3) 
in the previous theorem, we do not have to ask that those partially positive semidefinite kernels are of bounded shift type since 
this happens automatically.
\end{remark}

The last theorem answers the question: what happens if, in items (1) through (3) of Theorem~\ref{t:invariantukbp}, 
we ask that the partially positive semidefinite kernels $L$ be $\Gamma$-invariant?

\begin{theorem}\label{t:invariantukb}
Let $K\in \fK^h_{P_a}(X;\bH)$. The following assertions are equivalent.

\nr{1} There exists $L\in\fK^+_{P_a}(X;\bH)$, $\Gamma$-invariant and of bounded shift type, 
such that $-L\leq_{\mathrm{P_a}} K\leq_{\mathrm{P_a}} L$.

\nr{2} There exists $L\in\fK^+_{P_a}(X;\bH)$, $\Gamma$-invariant and of bounded shift type, 
such that, for each $s\in S$, we have
\begin{equation}\label{e:fegalp}
|[f,g]_K|\leq [f,f]_L^{1/2} [g,g]_L^{1/2},\quad f,g\in \cF_\FF^0(X_s;\bH).
\end{equation}

\nr{3} $K=K_1-K_2$ for some $K_1,K_2\in \fK^+_{P_a}(X;\bH)$, both of them $\Gamma$-invariant and of bounded shift type.

\nr{3}$^\prime$ $K=K_+-K_-$ for some disjoint $K_\pm\in \fK^+_{P_a}(X;\bH)$, both of them $\Gamma$-invariant and of 
bounded shift type.

\nr{4} There exists a triple $(\bK;\bV;\Psi)$ with the following properties.
\begin{itemize}
\item[(IKL1)] The pair $(\bK;\bV)$ is a Krein space minimal partial linearisation of $K$.
\item[(IKL2)] $(\mathrm{Id}_S;\Psi)$ is a $*$-representation of $\Gamma$ on $\Gamma_{\bK}$
such that, $\Psi(\alpha) V_x=V_{\alpha\cdot x}$ for all $x\in X$ and all $\alpha\in\Gamma_{a(x)}$.
\item[(IKL3)] There exists a distinguished bundle $\{J_s\}_{s\in S}$, where $J_s$ is a fundamental symmetry on $\cK_s$, for all 
$s\in S$, such that $\Psi(\alpha)J_{d(\alpha)}=J_{c(\alpha)}\Psi(\alpha)$ holds for all $\alpha\in\Gamma$. 
\end{itemize}

\nr{5} There exist a Krein space bundle $\bR=\{\cR_{K^s}\}_{s\in S}$ where, for each $s\in S$, 
$K^s=K|_{X_s\times X_s}$ and $\cR_{K^s}$ is 
a reproducing kernel $\bH_s$-valued Krein space with $\bH_s$-operator valued reproducing 
kernel $K^s$ with a distinguished fundamental 
symmetry $J_s$, where $\bH_s:=\{\cH_x\}_{x\in X_s}$, and a  $*$-representation $(\mathrm{Id}_S,\Phi)$, in the sense of 
Definition~\ref{d:serep}, where $\Phi\colon \Gamma\ra \Gamma_{\bR}$ and the $*$-semigrou\-poid $\Gamma_{\bR}$ is
defined as in Example~\ref{ex:segahaba}, such that, we have 
\begin{equation}\label{e:pipaj} \Phi(\alpha)J_{d(\alpha)}=J_{c(\alpha)}\Phi(\alpha),\quad \alpha\in \Gamma,
\end{equation}
and
\begin{equation}\label{e:fivakad}
K^{c(\alpha)}_{\alpha \cdot x}=\Phi(\alpha) K^{d(\alpha)}_x,\quad \alpha\in\Gamma,\ x\in X_{d(\alpha)}
\end{equation}
\end{theorem}

\begin{proof} From previous results, the equivalence of (1), (2), (3) and (3)$^\prime$ are clear. Also, the 
equivalence of item (4) and (5) is clear.

(2)$\Ra$(4). Let $L\in\fK^+_{P_a}(X;\bH)$ be $\Gamma$-invariant and of bounded shift type and 
such that, for each $s\in S$, the inequality \eqref{e:fegalp} holds. We use the constructions and the notation as in the proof of the 
implication (2)$\Ra$(4) in Theorem~\ref{t:invariantukbp}. For each $s\in S$, we consider the Gram operator $G_{K^s}$ of the 
indefinite inner product $[\cdot,\cdot]_{K^s}$ with respect to the positive semidefinite inner product 
$\langle\cdot,\cdot\rangle_{L^s}$ and let 
\begin{equation*}
G_{K^s}=J_{G_{K^s}} |G_{K^s}|,
\end{equation*}
be its polar decomposition with respect to the Hilbert space $(\cG_s;\langle\cdot,\cdot\rangle_{L^s})$, where 
$J_{G_{K^s}}$ is a selfadjoint partial isometry. By the construction of the Krein space induced by $G_{K^s}$, 
see Example~\ref{ex:kas}, $J_{G_{K^s}}$ induces a fundamental symmetry on $\cK_s$, that we denote by $J_s$. 
It only remains to prove that property (IKL3) holds. To see this, let us observe that, since both $K$ and $L$ are 
$\Gamma$-invariant, for each $\alpha\in\Gamma$, the shift operators $\Psi(\alpha)$ and $\Psi(\alpha^*)$ are adjoints one to 
each other, with respect to the pair of indefinite inner products $[\cdot,\cdot]_{\cK^{d(\alpha)}}$ and 
$[\cdot,\cdot]_{\cK^{c(\alpha)}}$, and the pair of positive definite inner products 
$\langle\cdot,\cdot\rangle_{J_{d(\alpha)}}$ and $\langle\cdot,\cdot\rangle_{J_{c(\alpha)}}$, which imply that 
$\Psi(\alpha)J_{d(\alpha)}=J_{c(\alpha)} \Psi(\alpha)$. From here we get that
\begin{equation*}
\Psi(\alpha^*)=\Psi(\alpha)^*=J_{d(\alpha)} \Psi(\alpha)^* J_{c(\alpha)},
\end{equation*}
from which, by taking the adjoints, we get $\Psi(\alpha)J_{d(\alpha)}=J_{c(\alpha)} \Psi(\alpha)$.

(4)$\Ra$(3)$^\prime$. We use the notation from the proof of the implication (4)$\Ra$(3)$^\prime$ in 
Theorem~\ref{t:invariantukbp}. With definition of $K_+$ as in \eqref{e:kpm}, for arbitrary $\alpha\in\Gamma$,
$x\in X_{d(\alpha)}$, and $y\in X_{c(\alpha)}$, we have
\begin{align*}
K_+(\alpha\cdot x,y) & =V(\alpha\cdot x)^* J_{c(\alpha)} V(y)=V(x)^* \Psi(\alpha)^*J_{c(\alpha)} \\
& =V(x)^* J_{d(\alpha)} \Psi(\alpha)^* 
V(y) = V(x)^*J_{d(\alpha)} V(\alpha^*\cdot y)=K_+(x,\alpha^*\cdot y).
\end{align*}
Similarly, we prove that $K_-(\alpha\cdot x,y)=K_-(x,\alpha^*\cdot y)$. Thus, both $K_+$ and $K_-$ are $\Gamma$-invariant.
\end{proof}
\bigskip

\textbf{Declarations}
\bigskip

\textbf{Funding.} The author did not receive support from any organisation for the submitted work.\medskip

\textbf{Competing Interest.} The authors have no relevant financial or non-financial interests to disclose.\medskip

\textbf{Data.} No data is associated with the research of the submitted work.\medskip

\textbf{Work.} The author did all the work and no tool of Artificial Intelligence was used at any stage.

\end{document}